\documentclass[12pt]{article}
\usepackage{blindtext}
\usepackage{hyperref}
\usepackage{setspace}
\usepackage{float}
\hypersetup{
	hidelinks=true,
	colorlinks=true,
	linkcolor=blue,
}
\usepackage[automark]{scrlayer-scrpage}
\usepackage[utf8]{inputenc}  

\usepackage{tocloft}
\usepackage{blindtext}

\usepackage{todonotes}

\usepackage{amsthm}
\usepackage{enumitem}
\usepackage[ruled,vlined]{algorithm2e}
\usepackage{amsmath}           
\usepackage{amsfonts}          
\usepackage{amssymb}
\usepackage{mathalfa}           
\usepackage{graphicx}          
\usepackage[T1]{fontenc}       
\usepackage{ae}                

\usepackage{lastpage}          
\usepackage[margin=10pt,font=small,labelfont=bf]{caption} 
\usepackage[T1]{fontenc}
\usepackage{enumitem}
\usepackage{stmaryrd}
\usepackage[arrow, matrix, curve]{xy}
\usepackage{listings}
\usepackage{color} 
\definecolor{mygreen}{RGB}{28,172,0} 
\definecolor{mylilas}{RGB}{170,55,241}

\usepackage{tikz}
\usetikzlibrary{knots,calc,shapes.geometric}
\usetikzlibrary{positioning}

\newcommand{\RR}{\mathbb{R}}

\newcommand{\NN}{\mathbb{N}}

\newcommand{\diag}{\text{diag}}

\newcommand{\ST}{\quad\text{s.t.}\quad}
\newcommand{\st}{\text{s.t.}}

\newcommand*{\norm}[1]{\left\Vert#1\right\Vert}

\newtheorem{prop}{Proposition}[section]
\newtheorem{thm}[prop]{Theorem}
\newtheorem{defn}[prop]{Definition}
\newtheorem{lem}[prop]{Lemma}

\theoremstyle{definition}

\usepackage{calrsfs}
\DeclareMathAlphabet{\pazocal}{OMS}{zplm}{m}{n}

\usepackage{pdfpages}
\usepackage{layout}
\usepackage{comment}

\begin{document}

\title{\bfseries\scshape The Sparse(st) Optimization Problem: Reformulations, Optimality, Stationarity, and
Numerical Results}

\date{October 6, 2022}

\author{Christian Kanzow\thanks{University of Würzburg, Institute of Mathematics, Campus Hubland Nord,
         Emil-Fischer-Str.\ 30, 97074 Würzburg, Germany; kanzow@mathematik.uni-wuerzburg.de} \and
         Alexandra Schwartz\thanks{Technical University of Dresden,
         Faculty of Mathematics, Zellescher Weg 12--14,  01069 Dresden, Germany;
     	 alexandra.schwartz@tu-dresden.de} \and 
         Felix Wei{\ss}\thanks{University of Würzburg, Institute of Mathematics, Campus Hubland Nord,
         Emil-Fischer-Str.\ 30, 97074 Würzburg, Germany; felix.weiss@mathematik.uni-wuerzburg.de}}

\maketitle

{
\noindent
\small\textbf{\abstractname.}
   We consider the sparse optimization problem with nonlinear constraints and an objective function, which is given by the sum of a general smooth mapping and an additional term defined by the $ \ell_0 $-quasi-norm.
   This term is used to obtain sparse solutions, but difficult to handle due to its nonconvexity and nonsmoothness (the sparsity-improving term is even discontinuous).
   The aim of this paper is to present two reformulations of this program as a smooth nonlinear program with complementarity-type constraints.
   We show that these programs are equivalent in terms of local and global minima and introduce a problem-tailored
   stationarity concept, which turns out to coincide with the standard KKT conditions of the two reformulated
   problems.
   In addition, a suitable constraint qualification as well as second-order conditions for the sparse optimization problem are investigated.
   These are then used to show that three Lagrange-Newton-type methods are locally fast convergent.
   Numerical results on different classes of test problems indicate that these methods can be used to drastically improve sparse solutions obtained by some other (globally convergent) methods for sparse optimization problems.
\par\addvspace{\baselineskip}
}

{
\noindent
\small\textbf{Keywords.}
   Sparse optimization; global minima; local minima;
   strong stationarity; Lagrange-Newton method; 
   quadratic convergence;
   B-subdifferential.
\par\addvspace{\baselineskip}
}


\section{Introduction}

The sparse(st) optimization problem considered in this paper is the constrained problem
\begin{equation}\label{SparseOpt}
   \min_x f(x) + \rho \lVert x \rVert_0 \quad \text{s.t.} \quad x \in X \tag{SPO},
\end{equation}
with a parameter $\rho>0$, a feasible set $X$ (usually) given by 
\begin{equation*}\label{FeasSet}
   X = \{ x \in \RR^n  \mid g(x) \le 0, \; h(x) = 0\}    
\end{equation*}
with (at least) continuous functions $f: \RR^n \to \RR $, $g: \RR^n \to \RR^m $, $h: \RR^n \to \RR^p$ and $ \lVert x \rVert_0 $ being the number of nonzero components $ x_i $ of the vector $ x $.
Following standard terminology, we call $ \lVert x \rVert_0 $ the $ \ell_0 $-norm throughout this manuscript though it is not a norm. 
Typical applications, where sparse solutions of a given optimization problem are required, include
compressed sensing for sparse representation of signals or image data, sparse portfolio selection problems, feature selection in classification learning, sparse regression or the sparse principal component analysis, see \cite[Section 2]{Tillmann-et-al-2021} for an overview and references.

Following \cite{LeThi-PhamDinh-Le-Vo-2014}, the solution methods for problems like \ref{SparseOpt} can be divided into the following three categories: (a) convex approximations, (b) nonconvex approximations, and (c) nonconvex exact
reformulations.

The most common convex approximation technique uses the $ \ell_1 $-norm instead of the $ \ell_0 $-norm in \ref{SparseOpt}.
An overview on such $\ell_1$-surrogate models, their advantages ans solution approaches can be found in \cite[Section 4.1]{Tillmann-et-al-2021}. 
Provided that $ f $ and $ X $ themselves are convex, the resulting optimization problem is convex (though
nonsmooth) and can therefore be solved by a variety of methods for convex optimization, see \cite{Beck2017}.
This approach is very popular, for example, in solving compressed sensing problems.
On the other hand, there exist prominent applications, where the $ \ell_1 $-norm provides absolutely no sparsity (like the portfolio optimization problem used in our numerical section).

This drawback leads to other sparsity improving terms that result in nonconvex approximation schemes.
A natural choice is to use the $ \ell_p $-quasi-norm for some $ p \in (0,1) $, which is no longer convex,
but still continuous, see \cite{Ghilli2017}.
Despite its nonconvexity, if there are no constraints (i.e., $ X = \RR^n $), the resulting problem can still be
solved relatively efficiently by a proximal-type method.
For additional constraints, one can apply an augmented Lagrangian-type method and use the proximal-type approach to solve the resulting (unconstrained) subproblems, see \cite{Chen2017,DeMarchi2022}.
In principle, these techniques can also be used for the $ \ell_0 $-norm, but the discontinuity still causes some trouble and typically leads to slowly convergent (proximal-type) gradient methods, see \cite{DeMarchi2022}.
Another method belonging to the class of nonconvex approximations is the penalty decomposition method \cite{Lu-Zhang-2012}, which introduces an additional variable and solves the resulting problem
by an alternating minimization technique.
Also the DC-type methods (DC = difference of convex) described in \cite{LeThi-PhamDinh-Le-Vo-2014} result in a nonconvex 
approximation which is shown to be exact under some additional assumptions, see also the DC-reformulation of the 
$ \ell_0 $-norm from \cite{Gotoh-Takeda-Tono-2018} (this reformulation, however, is applied to cardinality-constrained problems where the $ \ell_0 $-term is not in the objective function but in
the constraints, see below for a more detailed discussion).

Finally, regarding the class (c) of exact nonconvex reformulations, there are, to the best of our knowledge, still just a very few papers providing such reformulations.
A natural choice is to use a mixed-integer program, cf.\ reformulation \ref{MIP}.
This is useful for finding sparse solutions of -- often quadratic -- problems, whose dimension is not too large, and allows, in principle, to compute a global minimum, see e.g. \cite{ben2021global}.
By modifying the objective function with a suitable regularizing term, c.f. \cite{bertsimas2020sparse}, also larger problem dimensions can be handled. 
For nonlinear programs or large-scale problems, however, this typically leads to an intractable reformulation.
One alternative approach is the complementarity-type reformulation suggested in \cite{l0}, which can be shown to be 
completely equivalent to the original sparse optimization problem \ref{SparseOpt}.
The focus of the paper \cite{l0}, however, is slightly different.

More precisely, in this paper, we present two reformulations of the general sparse optimization problem \ref{SparseOpt}.
These reformulations are introduced in  Section~\ref{Sec:Reformulation}, and partially motivated by a related approach from \cite{CCSPO,CKS16} for cardinality-constrained optimization problems, cf.\ the corresponding discussion in Section~\ref{Sec:Reformulation}.
One of the two reformulations is exactly the one from \cite{l0} that we already mentioned previously.
Note that the subsequent results shown for our two reformulated problems are even new for the
approach from \cite{l0}.
In particular, we verify in Section~\ref{Sec:Properties} that problem \ref{SparseOpt} and our two reformulations
are equivalent in terms of both local and global minima.
Section~\ref{Sec:Properties} introduces a problem-tailored strong stationarity concept and a corresponding
constraint qualification and shows that these correspond to the standard KKT conditions and a standard constraint qualification of the two reformulated problems.
We then discuss suitably adapted second-order conditions in Section~\ref{Sec:SecondOrder}.

Though the main goal of this paper is to lay the foundations of two exact nonconvex reformulations of the sparse optimization problem \ref{SparseOpt}, the corresponding discussion leads, in a very natural way, to Lagrange-Newton-type methods for the solution  of \ref{SparseOpt}, see Section~\ref{Sec:LagrangeNewton}.
Like all Newton-type methods, this is primarily a locally (fast) convergent algorithm, whereas a central difficulty for 
the solution of sparse optimization problems is to design suitable globally convergent methods.
Nevertheless, the corresponding numerical results in Section~\ref{Sec:Numerics} indicate that the Lagrange-Newton-type methods can be used to obtain significant improvements over solutions calculated by other (globally convergent) sparse solvers.
We close with some final remarks in Section~\ref{Sec:Final}.

Notation: Throughout this manuscript, $ e_i \in \RR^n $ denotes the $ i $-th unit vector, whereas $ e := (1, \ldots, 1)^T \in \RR^n $ is the all-one vector. 
Given $ x \in \RR^n $ and $ x^* \in X $, we define the index sets
$$
   I_0(x) := \{ i \ | \ x_i = 0 \} \quad \text{and} \quad
   I_g(x^*) := \{ i \ | \ g_i(x^*) = 0 \}
$$
of zero components of $ x $ and active inequality constraints at $ x^* $, respectively.
For an arbitrary vector $ x $, we write $ \diag(x) $ for the corresponding diagonal matrix, whose diagonal entries are given by the elements of $ x $.
Given two vectors $ x, y \in \RR^n $, the Hadamard (elementwise) product is denoted by $ x \circ y $, 
i.e., the elements of this vector are given by $ x_i \cdot y_i $ for all $ i = 1, \ldots, n $.

\section{Two Smooth Reformulations of SPO}\label{Sec:Reformulation}

In this section we derive two smooth reformulations of \ref{SparseOpt} and show that the local and global minima of these reformulated problems coincide with the local and global minima of the original sparse optimization problem \ref{SparseOpt}.
One of these reformulations is already known from \cite{l0}, whereas the other one is new and will be more suited
for our numerical experiments later on.
Note that the results stated in this manuscript for the known formulation from \cite{l0} are still new and not contained in that reference.
Throughout this section, we only require $ f, g, h $ to be continuous.

Let us consider the sparse optimization problem from \ref{SparseOpt} with an arbitrary set $ X \subseteq \RR^n $.
For any $x \in \RR^n$, define a corresponding binary variable $y \in \{0,1\}^n$ by setting $ y_i := 0 $ for $ x_i \neq 0 $ and $ y_i := 1 $ for $ x_i = 0 $.
Using this $y$, we can calculate the $ \ell_0 $-norm of $x$ as
$$
   \|x\|_0 = \sum_{x_i \neq 0}1 = \sum_{i=1}^n (1-y_i) = n-e^Ty.
$$
Thus, we could rewrite problem \ref{SparseOpt} by the following mixed-integer problem
\begin{equation}\label{MIP}\tag{MIP}
   \min_{x,y} f(x) + \rho(n-e^Ty) \ST  x \in X, \; x \circ y=0, \;  y \in \{0,1\}^n.
\end{equation}
In order to move to a continuous optimization problem, we discard the binary constraints on $y$.
We need to retain the constraint $y \leq e$, because otherwise the objective function of \eqref{MIP} does not admit a minimum.
This leads us to the reformulation
\begin{equation}\label{PropPap}
   \min_{(x,y)} f(x) + \rho (n-e^Ty) \ST x \in X, \; x \circ y = 0, \; y \leq e.
   \tag{SPOlin}
\end{equation}
Since the auxiliary variable $ y $ enters the objective function linearly, we denote this
problem \ref{PropPap}. This is in contrast to our second formulation
\begin{equation}\label{ProbNew}
   \min_{(x,y)} f(x) + \frac{\rho}{2} \sum_{i=1}^n y_i (y_i - 2) \ST x \in X, \; x \circ y =0 
   \tag{SPOsq}
\end{equation}
called \ref{ProbNew}, since we add a quadratic term to the objective function.
Note that this quadratic term is designed in such a way that it vanishes, whenever $ x_i \neq 0 $ (due to the complementarity-type constraint), and that it attains its minimum at $ y_i = 1 $ whenever this variable is unconstrained, i.e., for all  $ i $ with $ x_i = 0 $., see Figure~\ref{fig:comparison}.

\begin{figure}[htbp]
   \centering
   \begin{tikzpicture}
      \draw[->] (-1.5,0) -- (-1,0) node[below]{-1} -- (0,0) node[below left]{0} -- (1,0) node[below]{1} -- (2,0) node[below]{2} -- (2.5,0) node[right]{$y_i$};
      \draw[->] (0,-1.5) -- (0,-1) node[left]{-1} -- (0,1) node[left]{1} -- (0,1.5);
      \draw[blue, very thick] (-1,1) -- (1.25,-1.25) node[right]{$-y_i$};
      \draw[dashed] (1,-1.5) -- (1,1.5);
   \end{tikzpicture}
   \qquad
   \begin{tikzpicture}
      \draw[->] (-1.5,0) -- (-1,0) node[below]{-1} -- (0,0) node[below left]{0} -- (1,0) node[below]{1} -- (2,0) node[below]{2} -- (2.5,0) node[right]{$y_i$};
      \draw[->] (0,-1.5) -- (0,-1) node[left]{-1} -- (0,1) node[left]{1} -- (0,1.5);
      \draw[-, blue, very thick, domain= -0.5:2.25, smooth, variable=\y] plot ({\y}, {\y * (\y-2)}) node[above]{$y_i (y_i - 2)$};
   \end{tikzpicture}
   \caption{Comparison of the terms $-y_i$ used in \ref{PropPap} and $y_i (y_i - 2)$ used in \ref{ProbNew}}\label{fig:comparison}
\end{figure}
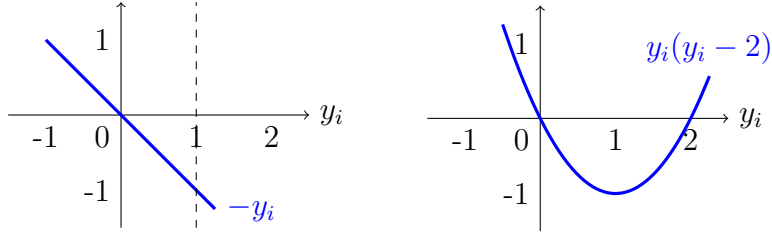

Problem \ref{PropPap} corresponds to the reformulation already introduced in \cite{l0}, whereas \ref{ProbNew} seems to be new.
Observe that, if the feasible set $ X $ contains no inequality constraints, then the new formulation \ref{ProbNew} boils down to an equality-constrained optimization problem, in contrast to \ref{PropPap}, which still includes the inequalities $ y \leq e $.
This observation is particularly useful in our setting since, later, we will apply a Lagrange-Newton-type method in order to solve the sparse optimization problem.

Before we take a closer look at the relaxed problems \ref{PropPap} and \ref{ProbNew}, we would like to briefly discuss the relation of the sparse problem \ref{SparseOpt} and its relaxations to the two closely related problem classes of \emph{cardinality-constrained problems}
\begin{equation*}
   \min_x \ f(x) \ST x \in X, \; \|x\|_0 \leq \kappa
\end{equation*}
and \emph{cardinality minimization problems}
\begin{equation*}
   \min_x \ \|x\|_0  \ST x \in X, \; f(x) \leq \delta,
\end{equation*}
where $\kappa \in \NN$ and $\delta \in \RR$ are given constants.
Using the same ideas as above, these problems can be relaxed to the continuous problems
\begin{eqnarray*}
   \min_{x,y} \ f(x) & \st & x \in X, \; x \circ y = 0, \; y \leq e, \; n - e^Ty \leq \kappa, \\
   \min_{x,y} \ n - e^Ty & \st & x \in X, \; x \circ y = 0, \; y \leq e, \; f(x) \leq \delta,
\end{eqnarray*}
respectively.
As we show below, for problem \ref{SparseOpt} the two relaxations are equivalent to the original problem in terms of global and local minima.
Using the same arguments, it is also possible to show this equivalence for the cardinality minimization problem.
However, for the cardinality-constrained problem it is known, see \cite{CCSPO}, that only the global minima of the original problem and its relaxation coincide, but the relaxation may have additional local minima.

Furthermore, one may be tempted to view problem \ref{SparseOpt} as a penalty reformulation of either of the other two problems.
However, while a solution $x^*$ of \ref{SparseOpt} always is a solution of the other two problems with $\kappa := \|x^*\|_0$ or $\delta := f(x^*)$, respectively, the opposite implication is in general not true.
This means that solutions of the cardinality-constrained problem or cardinality minimization problem cannot always be recovered as solutions of \ref{SparseOpt}.
More details on these relations can be found in \cite[Proposition 1.1]{Tillmann-et-al-2021}.

\section{Properties of Reformulations}\label{Sec:Properties}

In the moment, it is not clear why we can view the programs \ref{PropPap} and \ref{ProbNew} as reformulations of the given nonsmooth and discontinuous sparse optimization problem \ref{SparseOpt}.
But, as we show below, these three programs are completely equivalent in terms of both global and local
minima and even their corresponding stationary points coincide. 

In order to verify these statements, we first need some preliminary results.
Note that $ x $ is obviously feasible for the given problem \ref{SparseOpt} if and only if there exists a suitable vector $ y \in \RR^n $ such that $ (x,y) $ is feasible for \ref{PropPap} or \ref{ProbNew}.
Furthermore, we have the following relations for feasible points of these two programs.

\begin{lem}\label{0normprop}
   The following statements hold:
   \begin{enumerate}[label = (\roman*)]
      \item Let $(x,y)$ be feasible for \ref{PropPap}. Then
      $ \norm{x}_0 \le n - e^T y, $
      with equality if and only if $y_i = 1$ for all $i \in I_0(x)$.
      \item Let $(x,y)$ be feasible for \ref{ProbNew}. Then
      $ \norm{x}_0 - n \le \sum_{i=1}^n y_i (y_i - 2) , $
      with equality if and only if $y_i = 1$ for all $i \in I_0(x)$.
   \end{enumerate}
\end{lem}

\begin{proof}
   (i) The definition of the index set $ I_0(x) $ and the assumed feasibility of $ (x,y) $ implies
   \[
      n - e^T y = n - \sum_{i \in I_0(x)} y_i - \sum_{i \notin I_0(x)} y_i
      = n - \sum_{i \in I_0(x)} y_i
      \ge n -     \sum_{i \in I_0(x)} 1 = \norm{x}_0.
   \]
   This also shows that equality holds if and only if $ y_i = 1$ for all $i \in I_0(x)$. \smallskip

   \noindent
   (ii) Recall that the function $ y_i \mapsto y_i(y_i-2) $ attains its (unique) minimum at $ y_i = 1 $
   with corresponding minimal function value $ - 1 $. The definition of the index set $ I_0(x)$
   and the feasibility of $ (x,y) $ therefore yield
   \[
      \sum_{i=1}^n y_i (y_i - 2) = \sum_{i \in I_0(x)} y_i ( y_i - 2) \ge 
      \sum_{i \in I_0(x)} -1 = \Big( n - \sum_{i \in I_0(x)} 1 \Big) - n = \norm{x}_0 -n,
   \]
   and equality holds if and only if $y_i = 1$ for all $i \in I_0(x)$.
\end{proof}

\noindent
The following result shows that the constellation $y_i = 1$ for $i \in I_0(x)$ is indeed the most preferable one.

\begin{lem}\label{locmin}
   Let $(x^*,y^*)$ be a local minimum of \ref{PropPap} or \ref{ProbNew}.
   Then we have $y_i^* = 1$ for all $i \in I_0(x^*)$.
\end{lem}

\begin{proof}
   Let $(x^*,y^*)$ be a local minimum of \ref{PropPap}.
   We can fix $x=x^*$ and know that $y^*$ solves
   \[ 
      \max_y\ e^T y \ST  y_i =0, \ i\notin I_0(x^*), \; y \le e .
   \]
   Similarly, let $(x^*,y^*)$ be a local minimum of \ref{ProbNew}.
   We can fix $x=x^*$ and know that $y^*$ solves
   \[ 
      \min_y \ \sum_{i=1}^n y_i (y_i - 2) \ST  y_i=0, \ i\notin I_0(x^*).
   \]
   In both cases the statement follows.
\end{proof}

\noindent
Next, we show that the set of local minima of the sparse optimization problem \ref{SparseOpt} is independent of the particular choice of the penalty parameter.
Note that, this is due to the discontinuity of the $ \ell_0 $-norm and that a similar result for sparse optimization problems involving the $ \ell_1 $-norm, e.g., does not hold.
This observation may actually be viewed as an advantage of the $ \ell_0 $-norm, since this implies that a suitable choice
of the penalty parameter is much less critical for the $ \ell_0 $-formulation of the sparse optimization problem than other (continuous) formulations like the one based on the $ \ell_1 $-norm or the $ \ell_q $-quasi-norm for $ q \in (0,1) $.

\begin{prop}\label{Prop:PenIndependent}
   Let $ x^* $ be a local minimum of \ref{SparseOpt} with penalty parameter $ \rho_1 > 0 $.
   Then $ x^* $ is also a local minimum of \ref{SparseOpt} for any other penalty parameter $ \rho_2> 0 $.
\end{prop}

\begin{proof}
   Let $\rho_1$ and $\rho_2$ be two penalty parameters, and let $x^*$ be a local minimum of
   \begin{equation}\label{Eq:SPOrho1}
      \min_x \ f(x) + \rho_1 \norm{x}_0 \ST x\in X.
   \end{equation}
   Assume that $ x^* $ is not a local minimum of 
   \begin{equation*}
      \min_x \ f(x) + \rho_2 \norm{x}_0 \ST x\in X.
   \end{equation*}
   Then there exists a sequence $ \{ x^k \} \subseteq X $ with $ x^k \to x^* $ such that
   \begin{equation}\label{Eq:SPOrho2}
      f(x^k) + \rho_2 \norm{x^k}_0 < f(x^*) + \rho_2 \norm{x^*}_0 \quad \forall k \in \mathbb{N}.
   \end{equation}
   Note that $ \norm{x^k}_0 \geq \norm{x^*}_0 $ holds for all $ k $ sufficiently large.
   First consider the case that there exists a subsequence such that $ \norm{x^k}_0 = \norm{x^*}_0 $ holds for all $ k \in K $.
   Then we obtain
   \begin{eqnarray*}
      f(x^k) + \rho_1 \norm{x^k}_0 
      & = & f(x^k) + \rho_2 \norm{x^k}_0 + ( \rho_1 - \rho_2 ) \norm{x^k}_0 \\
      & < & f(x^*) + \rho_2 \norm{x^*}_0 + ( \rho_1 - \rho_2 ) \norm{x^k}_0 \\
      & = & f(x^*) + \rho_2 \norm{x^*}_0 + ( \rho_1 - \rho_2 ) \norm{x^*}_0 
      \ = \ f(x^*) + \rho_1 \norm{x^*}_0 
   \end{eqnarray*}
   for all $ k \in K $, contradicting the assumption that $ x^* $ is a local minimum of \eqref{Eq:SPOrho1}.
   In the other case, we have $ \norm{x^k}_0 > \norm{x^*}_0 $ and, therefore, $ \norm{x^*}_0 + 1 \leq \norm{x^k}_0 $ for almost all $ k \in \mathbb{N} $.
   Furthermore, by continuity of $ f $, it follows that $ f(x^*) \leq f(x^k) + \rho_2 $ for all $ k $ sufficiently large.
   This implies
   \begin{equation*}
      f(x^*) + \rho_2 \norm{x^*}_0 
      \leq f(x^k) + \rho_2 + \rho_2 \norm{x^*}_0 
      = f(x^k) + \rho_2 \big( 1 + \norm{x^*}_0 \big) 
      \leq f(x^k) + \rho_2 \norm{x^k}_0,
   \end{equation*}
   a contradiction to \eqref{Eq:SPOrho2}.
   Altogether, this completes the proof.
\end{proof}

\noindent
The previous statement also holds for the two reformulated programs \ref{PropPap} and
\ref{ProbNew}.
This is a consequence, e.g., of the following result, which states that $ x^* $ is a local minimum of the sparse optimization problem \ref{SparseOpt} if and only if there exists a vector $ y^* $ such that the pair $ (x^*, y^*) $ is a local minimum of either \ref{PropPap} or \ref{ProbNew}.

\begin{thm}[Equivalence of Local Minima]\label{EquivLocMin}
   The following statements are equivalent:
   \begin{enumerate}[label = (\roman*)]
      \item $x^*$ is a local optimum of \ref{SparseOpt}.
      \item There exists $y^*$ such that $(x^*,y^*)$ is a local optimum of \ref{PropPap}.
      \item There exists $y^*$ such that $(x^*,y^*)$ is a local optimum of \ref{ProbNew}.
   \end{enumerate}
\end{thm}

\begin{proof}
   Notice that, by Lemma~\ref{locmin}, $y^*$ has to be of the form 
   \[
      y^*_i = \begin{cases} 1 \quad \text{for} \ i \in I_0(x^*), \\ 0 \quad \text{otherwise,}\end{cases} \tag{*}
   \]
   in order for $(x^*,y^*)$ to be a local minimum of \ref{PropPap} or \ref{ProbNew}.\smallskip

   \noindent
   $(i) \Longrightarrow (ii)$: Let $x^*$ be a local minimum of \ref{SparseOpt} and let $y^*$ be defined as in (*). Then
   \begin{equation*} 
      f(x^*) + \rho\big(n - e^Ty^*\big) = f(x^*) + \rho \norm{x^*}_0
      \le f(x) +\rho \norm{x}_0
      \le f(x) + \rho \big(n-e^Ty\big)
   \end{equation*}
   for all feasible $(x,y)$ with $x$ sufficiently close to $x^*$, where the first equality and the
   last inequality follow from Lemma~\ref{0normprop}$(i)$.
   \smallskip

   \noindent
   $ (ii) \Longrightarrow (i) $: Let $ (x^*,y^*) $ be the local minimum of \ref{PropPap}
   with $ y^* $ as in (*). Assume that $ x^* $ is not a local minimum of
   \ref{SparseOpt}. Then there exists a sequence $ \{ x^k \} \subseteq X $ such that
   $ x^k \to x^* $ and
   \begin{equation}\label{Eq:ContraLocal}
      f(x^k) + \rho \norm{x^k}_0 < f(x^*) + \rho \norm{x^*}_0 \quad
      \forall k \in \mathbb{N}.
   \end{equation}
   Recall that $ \norm{x^k}_0 \geq \norm{x^*}_0 $ holds for all $ k $ sufficiently large.
   Hence we either have a subsequence $ \{ x^k \}_K $ such that $ \norm{x^k}_0 = \norm{x^*}_0 $
   holds for all $ k \in K $, or $ \norm{x^*}_0 + 1 \leq \norm{x^k}_0 $ is true for almost
   all $ k \in \mathbb{N} $. In the former case, it follows that $ (x^k, y^*) $ is feasible
   for \ref{PropPap}, hence we obtain from Lemma~\ref{0normprop}$(i)$ and the minimality of
   $ (x^*, y^*) $ for \ref{PropPap} that
   \begin{equation*}
      f(x^k) + \rho \norm{x^k}_0 = f(x^k) + \rho \norm{x^*}_0
      = f(x^k) + \rho (n - e^T y^*)
      \geq f(x^*) + \rho (n - e^T y^*)
      = f(x^*) + \rho \norm{x^*}_0,
   \end{equation*}
   which contradicts \eqref{Eq:ContraLocal}. Otherwise, we have $ \norm{x^*}_0 + 1 \leq \norm{x^k}_0 $ 
   and, by continuity, also $ f(x^*) \leq f(x^k) + \rho $ for all $ k \in \mathbb{N} $ sufficiently
   large, which, in turn, gives
   \begin{equation*}
      f(x^k) + \rho \norm{x^k}_0 \geq f(x^k) + \rho + \rho \norm{x^*}_0 \geq 
      f(x^*) + \rho \norm{x^*}_0.
   \end{equation*}
   Hence, also in this situation, we have a contradiction to \eqref{Eq:ContraLocal}.\smallskip

   \noindent
   $ (i) \Longrightarrow (iii) $: Let $ x^* $ be a local minimum of \ref{SparseOpt}. Then
   $ x^* $ is also a local minimum of the optimization problem
   \begin{equation}\label{Eq:SPOhalf}
      \min \ f(x) + \frac{\rho}{2} \big( \norm{x}_0 - n \big) \quad \text{s.t.} \quad x \in X,
   \end{equation}
   since, by Proposition~\ref{Prop:PenIndependent}, we can modify the penalty parameter, and since adding a constant to the objective function does not change the location of the local minima.
   Now, let $ y^* $ be defined as in statement (*). Then
   \begin{equation*} 
      f(x^*) + \frac{\rho}{2} \sum_{i=1}^n y_i^*\big(y_i^* - 2\big) = f(x^*) + \frac{\rho}{2} \big( \norm{x^*}_0 - n \big)
      \le f(x) + \frac{\rho}{2} \big( \norm{x}_0 - n \big)
      \le f(x) + \frac{\rho}{2} \sum_{i=1}^n y_i\big(y_i - 2\big),
   \end{equation*}
   for all feasible $(x,y)$ with $x$ sufficiently close to $x^*$, where the first equality and the last inequality follow from Lemma~\ref{0normprop}$(ii)$.\smallskip

   \noindent
   $ (iii) \Longrightarrow (i) $: Let $ (x^*,y^*) $ be a local minimum of \ref{ProbNew}
   with $ y^* $ as in (*). Assume that $ x^* $ is not a local minimum of
   \ref{SparseOpt}. Then $ x^* $ is not a local minimum of \eqref{Eq:SPOhalf}. Hence, there exists a sequence 
   $ \{ x^k \} \subseteq X $ such that $ x^k \to x^* $ and
   \begin{equation}\label{Eq:ContraLocalSPOhalf}
      f(x^k) + \frac{\rho}{2} \big( \norm{x^k}_0 - n\big) < f(x^*) + \frac{\rho}{2} \big( \norm{x^*}_0 - n \big) \quad
      \forall k \in \mathbb{N}.
   \end{equation}
   Recall that $ \norm{x^k}_0 \geq \norm{x^*}_0 $ holds for all $ k $ sufficiently large.
   Thus, once again, we either have a subsequence $ \{ x^k \}_K $ such that $ \norm{x^k}_0 = \norm{x^*}_0 $
   holds for all $ k \in K $, or $ \norm{x^*}_0 + 1 \leq \norm{x^k}_0 $ is true for almost
   all $ k \in \mathbb{N} $. In the former case, it follows that $ (x^k, y^*) $ is feasible
   for \ref{ProbNew}, hence we obtain from Lemma~\ref{0normprop}$(ii)$ and the minimality of
   $ (x^*, y^*) $ for \ref{ProbNew} that
   \begin{eqnarray*}
      f(x^k) + \frac{\rho}{2} \big(\norm{x^k}_0 - n\big) & = & f(x^k) + \frac{\rho}{2} \big( \norm{x^*}_0 - n\big)
      \ = \ f(x^k) + \frac{\rho}{2} \sum_{k=1}^n y_i^*\big( y_i^* - 2 \big) \\
      & \geq & f(x^*) + \frac{\rho}{2} \sum_{k=1}^n y_i^*\big( y_i^* - 2 \big)
      \ = \ f(x^*) + \frac{\rho}{2} \big(\norm{x^*}_0 - n\big),
   \end{eqnarray*}
   which contradicts \eqref{Eq:ContraLocalSPOhalf}. Otherwise, we have $ \norm{x^*}_0 + 1 \leq \norm{x^k}_0 $ 
   and, by continuity, also $ f(x^*) \leq f(x^k) + \frac{\rho}{2} $ for all $ k \in \mathbb{N} $ sufficiently
   large, which, in turn, gives
   \begin{equation*}
      f(x^k) + \frac{\rho}{2} \big(\norm{x^k}_0-n\big) \geq f(x^k) + \frac{\rho}{2} + \frac{\rho}{2} \big(\norm{x^*}_0 - n\big) \geq 
      f(x^*) + \frac{\rho}{2} \big(\norm{x^*}_0-n\big).
   \end{equation*}
   Hence, also in this situation, we have a contradiction to \eqref{Eq:ContraLocalSPOhalf}.
\end{proof}

\noindent
Scaling the penalty parameter $\rho$ as in the proof of the previous result has, of course, an impact on the global minima of \ref{SparseOpt}.
We therefore do not obtain equivalence of the global minima in the above sense, i.e., independent of the choice of the penalty parameter.
However, the following result holds.

\begin{thm}[Equivalence of Global Minima]\label{thm:GlobMin} 
   The following statements hold:
   \begin{enumerate}[label = (\roman*)]
      \item $x^*$ is a global minimum of \ref{SparseOpt} if and only if 
      there exists $y^*$ such that $(x^*,y^*)$ is a global minimum of \ref{PropPap}.
      \item $x^*$ is a global minimum of \ref{SparseOpt} with penalty parameter $\frac{\rho}{2}$ 
      if and only if there exists $y^*$ such that $(x^*,y^*)$ is a global minimum of \ref{ProbNew}.
   \end{enumerate}
\end{thm}

\begin{proof}
   According to Lemma~\ref{0normprop} $(i)$, the inequality $ f(x) + \rho (n-e^Ty) \ge f(x) + \rho \norm{x}_0 $ holds for all $(x,y)$ feasible for \ref{PropPap}, with equality if and only if $y_i = 1$ for all $i \in I_0(x)$.
   The pair $(x^*,y^*)$ therefore solves \ref{PropPap} if and only if $x^*$ solves \ref{SparseOpt}, with $y_i^* = 1 $ for all $i \in I_0(x^*)$.

   To prove part $(ii)$, we recall that $x^*$ is a global minimum of \ref{SparseOpt} with penalty parameter $\frac{\rho}{2} $ if and only if $x^*$ is a solution of
   \[ 
      \min_x \ f(x) + \frac{\rho}{2} \big( \norm{x}_0 - n \big) \ST x\in X. 
   \]
   Using Lemma~\ref{0normprop} $(ii)$, the claim follows analogously to the proof of part $(i)$.
\end{proof}

\noindent
Effectively, formulation \ref{ProbNew} can be considered as a reformulation of the scaled problem 
\[ 
   \min_x \ f(x) + \frac{\rho}{2} \norm{x}_0 \ST x\in X . 
\]
Nevertheless, invariance of the local minima to the chosen parameter $\rho$ is also reflected in the stationary conditions, which we derive in the next section.
We therefore neglect the scaling issue in our subsequent analysis of a local Newton-type method, as any solution found cannot guaranteed to be globally optimal.

\section{Stationary Conditions}\label{Sec:Stationarity}

This section introduces a stationarity concept for the nonsmooth and discontinuous sparse
optimization problem \ref{SparseOpt} and relates it to the KKT conditions of the two
smooth reformulations from \ref{PropPap} and \ref{ProbNew}.
Throughout this section, 
we assume that all functions $ f, g, h $ are continuously differentiable.

To this end, let us introduce the function
\[ 
   L^{SP}(x,\lambda, \mu) := f(x) + \lambda^T g(x) + \mu^T h(x)
\]
which is exactly the Lagrangian of \ref{SparseOpt} except that we do not include
the term with the $ \ell_0 $-norm.
In particular, $ L^{SP} $ is therefore a smooth function.
Based on $ L^{SP} $, the ordinary Lagrangians of the smooth optimization
problems \ref{PropPap} and \ref{ProbNew} can be written as
\[ 
   L^{lin}(x,y,\lambda, \mu, \gamma,\sigma) := L^{SP}(x,\lambda, \mu) + \rho (n-e^Ty) + 
   \gamma^T (x \circ y) + \sigma^T (y - e)  
\]
and
\[ 
  L^{sq}(x,y,\lambda,\mu,\gamma):= L^{SP}(x,\lambda,\mu) + 
  \frac{\rho}{2} \sum_{i=1}^n y_i (y_i - 2) + \gamma^T (x \circ y) , 
\]
respectively.
The standard KKT conditions of \ref{PropPap} are therefore given by
\begin{align}
   \nabla_x L^{lin} (x,y,\lambda,\mu,\gamma,\sigma)
   = \nabla_x L^{SP}(x,\lambda, \mu) + \gamma \circ y &= 0, \label{KKTs}\\
   \nabla_y L^{lin} (x,y,\lambda,\mu,\gamma,\sigma)
   = -\rho e + \gamma \circ x + \sigma &= 0, \label{SigInd}\\
   \lambda \ge 0, \quad g(x) \le 0, \quad \lambda \circ g(x) &= 0, \label{UnglRes}\\
   h(x) &= 0, \label{GlhRes}\\
   x \circ y &= 0, \label{Ortho} \\
   \sigma \ge 0, \quad y \le e, \quad \sigma \circ (y-e) &= 0. \label{eqn:NCP}
\end{align}
We take a closer look at system (\ref{SigInd}), (\ref{Ortho}), (\ref{eqn:NCP}) componentwise for $i = 1,...,n$
\begin{align}
    -\rho + \gamma_i x_i + \sigma_i &= 0, \label{SigIndcomp} \\
    x_i \cdot y_i = 0, \label{Ortho:comp}\\
    \sigma_i \ge 0, \quad y_i \le 1, \quad \sigma_i(y_i - 1) &= 0, \label{eqn:NCPcomp}
\end{align}
and assume there is a solution $(x_i^*,y_i^*,\gamma_i^*,\sigma_i^*)$. We distinguish two cases. First, let $x_i^* = 0$, then clearly $\sigma_i^* = \rho$ and $y_i^* = 1$, whereas $\gamma_i^*$ is arbitrary. In the second case, we have $x_i^* \neq 0$, which immediately implies $y_i^* = 0$, $\sigma_i^* = 0$ and further $\gamma_i^* = \rho/x_i^*$. Hence, $(x_i^*,y_i^*,\gamma_i^*)$ also solves
\begin{align}
    \rho(y_i - 1) +  \gamma_i x_i = 0 \quad \text{and} \quad x_i \cdot y_i = 0. \label{eqn:reduced}
\end{align}
Conversely, let $(x_i^*,y_i^*,\gamma_i^*)$ be a solution of equation (\ref{eqn:reduced}). Then with $\sigma_i^*=\rho$, if $x_i^* = 0$ and $\sigma_i^* = 0$, if $x_i^* \neq 0$ the tuple $(x_i^*,y_i^*,\gamma_i^*,\sigma_i^*)$ is clearly a solution of system (\ref{SigIndcomp}), (\ref{Ortho:comp}), (\ref{eqn:NCPcomp}).

Using this reasoning, we can compress the system (\ref{KKTs})--(\ref{eqn:NCP}) by deleting the variable $\sigma$ to the system
\begin{align}
   \nabla_x L^{SP}(x,\lambda,\mu) + \gamma \circ y &= 0, \label{KKTstart}\\
   \rho (y - e) + \gamma \circ x &= 0, \label{KKTnew} \\
   \lambda \ge 0, \quad g(x) \le 0, \quad \lambda \circ g(x) &= 0. \label{KKTe} \\
   h(x) &= 0, \\
   x \circ y &= 0, \label{Ortho2} 
\end{align}
Now, it is easy to see that \eqref{KKTstart}--\eqref{KKTe} are precisely the KKT conditions of
problem \ref{ProbNew}. In summary, we have the following result.

\begin{prop}[Equivalence of KKT Points]\label{EquivKKT}
   The vector $(x^*,y^*,\lambda^*,\mu^*,\gamma^*)$ is a KKT point of \ref{ProbNew} if and only
   if there exists $\sigma^*$ such that $(x^*,y^*,\lambda^*,\mu^*,\gamma^*,\sigma^*)$ is a KKT point of \ref{PropPap}. The multipliers $\sigma^*$, $\gamma^*$ and the variable $y^*$ depend uniquely on 
   $(x^*,\lambda^*,\mu^*)$ with
   \begin{equation*}
      y_i^* = \begin{cases} 1, & i\in I_0(x^*), \\ 0, & i\notin I_0(x^*), \end{cases}\quad
      \gamma_i^* =\begin{cases} - \nabla_{x_i} L^{SP}(x^*,\lambda^*,\mu^*)\ , & i\in I_0(x^*), \\
      \frac{\rho}{x_i^*} \ , & i\notin I_0(x^*),\end{cases}\quad
      \sigma_i^* = \begin{cases} \rho , & i\in I_0(x^*), \\ 0 , & i\notin I_0(x^*). \end{cases}
   \end{equation*}
\end{prop}

\begin{proof}
   The equivalence of the two KKT systems is an immediate result by the equivalence of (\ref{SigInd}) and (\ref{eqn:NCP}) to (\ref{KKTnew}) under the condition (\ref{Ortho}) present in both systems, which we established componentwise.
   Additionally, we already verified the unique dependence of $y^*$ and $\sigma^*$ on $x^*$, as well as $\gamma^*_i = \rho/x_i^*$ for $i \notin I_0(x^*)$.
   The representation of $\gamma_i^*$ for $i \in I_0(x^*)$, on the other hand, can be obtained by (\ref{KKTs}).
\end{proof}

\noindent
For a fixed triple $(x,\lambda,\mu)$, the only possible choice of $(y,\gamma,\sigma)$ with which 
a KKT point of either of the above systems could be obtained, is therefore already determined. 
This, in turn, tells us that the possibility to satisfy the KKT conditions depends on the 
values of $(x, \lambda, \mu)$ only. This motivates to define a stationary concept for
the original sparse optimization problem \ref{SparseOpt} in the following way.

\begin{defn}\label{Def:S-stationary}
   We call a point $x^*$ an {\em S-stationary point} ({\em strongly stationary point}) 
   of \ref{SparseOpt} if there exist multipliers $(\lambda^*, \mu^*)$ such that 
   the following conditions hold:
   \begin{align}\label{SS}
      \begin{split}
         \nabla_{x_i} L^{SP} (x^*,\lambda^*, \mu^*) &= 0, \quad \forall i \notin I_0(x^*), \\
         \lambda^* \ge 0, \ g(x^*) \le 0, \ \lambda^* \circ g(x^*) &=0, \\
         h(x^*) &= 0.
      \end{split}
   \end{align}
\end{defn}

\noindent
Note that there exist a couple of different stationarity concepts like 
W-, C-, M-, and S-stationarity for a number of related problem classes,
including mathematical programs with complementarity constraints \cite{HKS16},
cardinality constraints \cite{CKS16}, vanishing constraints \cite{HoK08},
and switching constraints \cite{Meh20}. 
Similarly, it would be possible to state some of these other stationarity concepts
for problem \ref{SparseOpt} as well. However, on the one hand, it turns out 
that suitable methods for the solution of sparse optimization problems can
be shown to converge to S-stationary points, see \cite{PhD-ThesisRaharja} for some preliminary
results in this direction, which is in contrast to the other classes of
problems mentioned before and which indicates that there is no need to introduce
these weaker stationarity concepts for sparse optimization problems, and, on the other
hand, for the purpose of the approach presented here, we only require the
S-stationarity from Definition~\ref{Def:S-stationary}.

S-stationarity turns out to be equivalent to the KKT conditions of the reformulated problems
\ref{PropPap} and \ref{ProbNew}.

\begin{thm}[Equivalence of S-Stationary and KKT Points]\label{Thm:EquivStat}
   The following are equivalent:
   \begin{enumerate}[label = (\roman*)]
      \item $x^*$ is S-stationary for \ref{SparseOpt} with some multipliers $(\lambda^*, \mu^*)$.
      \item There exists $(y^*,\gamma^*, \sigma^*)$, depending on $(x^*,\lambda^*,\mu^*)$ only, such that
      $(x^*,y^*,\lambda^*,\mu^*,\gamma^*,\sigma^*)$ is a KKT point of \ref{PropPap}.
      \item There exists $(y^*,\gamma^*)$, depending on $(x^*,\lambda^*,\mu^*)$ only, such that
      $(x^*,y^*,\lambda^*,\mu^*,\gamma^*)$ is a KKT point of \ref{ProbNew}.
   \end{enumerate}
\end{thm}

\begin{proof}
   Assume $x^*$ is S-stationary for \ref{SparseOpt}.
   Then there exists $(\lambda^*,\mu^*)$ such that \eqref{SS} holds.
   Choosing $y^*$ and $\gamma^*$ as in Proposition~\ref{EquivKKT}, we obtain a KKT point of \ref{ProbNew}.
   Conversely, let $(x^*,y^*,\lambda^*,\mu^*,\gamma^*)$ be a KKT point of \ref{ProbNew}.
   Then \eqref{KKTstart} holds.
   Hence \eqref{SS} is satisfied for $(x^*,\lambda^*,\mu^*)$, which implies that $x^*$ is an S-stationary point of \ref{SparseOpt}.
   The remaining equivalence follows from Proposition~\ref{EquivKKT}.
\end{proof}

\noindent
We next introduce a problem-tailored constraint qualification which, in particular, 
guarantees that a local minimum of \ref{SparseOpt} is an S-stationary point. This
constraint qualification is relatively strong, and much weaker ones will be discussed
in a forthcoming report. For the purpose of this paper, where we plan to consider
a Lagrange-Newton-type method for the solution of sparse optimization problems, the
following condition is the most suitable one.

\begin{defn}
   A feasible point $x^* \in X$ of \ref{SparseOpt} satisfies the \emph{sparse LICQ} (\emph{ SP-LICQ}, for
   short) if the vectors
   \[ 
      \nabla g_i(x^*) \ (i \in I_g(x^*)), \quad \nabla h_i(x^*) \ (i = 1,...,p),
      \quad e_i \ (i \in I_0(x^*))
   \]
   are linearly independent.
\end{defn}

\noindent
Note that SP-LICQ corresponds to standard LICQ of the {\em tightened
nonlinear program}
\begin{equation}\label{Eq:TNLP}
   \min_x \ f(x) \ST g(x) \leq 0, \; h(x) = 0, \; x_i = 0 \ (i \in I_0(x^*))
\end{equation}
depending on a feasible point $ x^* \in X $. 
We establish the following connection between SP-LICQ for \ref{SparseOpt} with standard LICQ for \ref{PropPap} and \ref{ProbNew}.

\begin{thm}[Equivalence of LICQ-type Conditions]\label{EquivLICQ} 
   Let $(x^*,y^*)$ be feasible for \ref{PropPap} and \ref{ProbNew}, respectively and assume
   $\{i \ | \ x_i^* = y_i^* = 0\} = \emptyset$.
   Then the following are equivalent:
   \begin{enumerate}[label = (\roman*)]
      \item SP-LICQ is satisfied at $x^*$,
      \item Standard LICQ holds at $(x^*,y^*)$ for \ref{PropPap},
      \item Standard LICQ holds at $(x^*,y^*)$ for Problem \ref{ProbNew}.
   \end{enumerate}
\end{thm}

\begin{proof}
It is easy to see that SP-LICQ holds at $x^*$  for \ref{SparseOpt} if and only if 
the following vectors are linearly independent:
\begin{align} 
\begin{split}
   \begin{pmatrix}\nabla g_i(x^*) \\ 0\end{pmatrix} \ (i \in I_g(x^*)), \quad
   \begin{pmatrix}\nabla h_i(x^*) \\ 0 \end{pmatrix} \ (i = 1,...,p), \quad 
   \begin{pmatrix}\alpha_i e_i \\ 0\end{pmatrix} \ (i \in I_0(x^*)),\\
   \begin{pmatrix} 0 \\ \beta_i e_i\end{pmatrix} \ (i \notin I_0(x^*)) \quad \begin{pmatrix}0 \\ \xi_i e_i\end{pmatrix} \ (i \in J),
\end{split}\label{gradients}
\end{align}
for arbitrary $\alpha_i, \beta_i, \xi_i \in \RR \setminus\{0\}$ and an arbitrary 
subset $J \subseteq I_0(x^*)$.\\

\noindent
{\em Case 1}: Choose $(x^*,y^*)$ feasible for \ref{PropPap} with $\{i \ | \ x_i^* = y_i^* = 0\} = \emptyset$. Now, set
$\alpha_i := y_i^*$ for $i \in I_0(x^*)$, $\beta_i := x^*_i$ for $i \notin I_0(x^*)$. 
Furthermore, set 
$J := \{ i\ | \ y_i^* = 1 \} \subset I_0(x^*)$ and $\xi_i := 1$ for $i \in J$, respectively. 
Plugging our choices of $\alpha, \beta, \xi$,
and $J$ into (\ref{gradients}) yields the set of gradients of the equality and active inequality constraints of \ref{PropPap}. The
claim follows.\\

\noindent
{\em Case 2}: Let $(x^*,y^*)$ feasible for \ref{ProbNew} with $\{i \ | \ x_i^* = y_i^* = 0\} = \emptyset$.
Choose $J := \emptyset$ and $\alpha, \beta$ as in case 1. Then system (\ref{gradients}) collapses to the set of gradients
of the equality and active inequality constraints of \ref{ProbNew}. The claim follows.
\end{proof}

\noindent
The central assumption in Theorem~\ref{EquivLICQ} is, of course, that the bi-active set
$\{i \ | \ x_i^* = y_i^* = 0\} $ is empty. In the context of sparse optimization 
problems and our reformulations, however, this assumption turns out to be very weak and
is automatically satisfied, for example, if $ x^* $ is a local minimum of \ref{SparseOpt} or at a KKT-point of either \ref{PropPap} or \ref{ProbNew}. 
This is an immediate consequence of Lemma~\ref{locmin}.

Therefore, if SP-LICQ holds at a local optimum $x^*$ of \ref{SparseOpt}, it follows 
that there is a unique vector $y^*$ with $y_i^* = 1$ for all $i \in I_0(x^*)$ such that 
the KKT conditions of \ref{PropPap} and \ref{ProbNew}, respectively, have a unique solution 
guaranteed by standard LICQ, which holds for both of the smooth reformulations. In 
particular, local minima of $x^*$ of \ref{SparseOpt}, where SP-LICQ holds, are thus S-stationary with uniquely defined multipliers. Nevertheless, SP-LICQ is a relatively strong
constraint qualification, and we will come back to this point later.

\section{Second-Order Conditions}\label{Sec:SecondOrder}

The aim of this section is to introduce problem-tailored second-order conditions for
the sparse optimization problem \ref{SparseOpt} and to relate these conditions to standard
second-order conditions associated with the two smooth reformulations 
\ref{PropPap} and \ref{ProbNew}, respectively. Naturally, these second-order conditions play a
central role for our subsequent development of Lagrange-Newton-type methods for 
the solution of sparse optimization problems. Note that, throughout this section,
we make the implicit assumption that all functions $f, g, h $ are twice continuously
differentiable.

\begin{defn}\label{Def:SOSC}
Let $x^*$ be an S-stationary point of \ref{SparseOpt}, with multipliers $(\lambda^*,\mu^*)$. We call
\begin{align*}
   C^{SPO}(x^*,\lambda^*) := \{ d  \mid \nabla g_i (x^*)^T d &= 0 \quad \forall i \in I_g(x^*),\; \lambda_i^* > 0, & \\
   \nabla g_i(x^*)^T d & \le 0 \quad \forall i \in I_g(x^*), \; \lambda_i^* = 0, &\\ 
   \nabla h(x^*)^T d &= 0, &\\
   d_i &= 0 \quad \forall i \in I_0(x^*) \qquad \qquad \},
   \end{align*}
and
\begin{align*}
   SC^{SPO}(x^*,\lambda^*) := \{ d \mid \nabla g_i (x^*)^T d &= 0 \quad \forall i \in I_g(x^*),\ \lambda_i^* > 0, \\
   \nabla h(x^*)^T d &= 0, \\
   d_i &= 0 \quad \forall i \in I_0(x^*) \qquad \qquad  \},
\end{align*}
the {\em critical cone} and {\em critical subspace}, respectively, of \ref{SparseOpt} at $x^*$ with multiplier $\lambda^*$.
\end{defn}

\noindent
Note that the critical cone and the critical subspace of problem \ref{SparseOpt}
are problem-tailored definitions, which can also be interpreted as the standard
critical cone and the standard critical subspace of the corresponding tightened nonlinear
program from \eqref{Eq:TNLP}. The usual critical cone and critical subspace of
problem \ref{SparseOpt} would not contain the condition that $ e_i ^T d = 0$ for 
$i \in I_0(x^*)$ and, hence, these standard sets would be larger than those from the previous
definition.

Definition~\ref{Def:SOSC} allows the following formulation of sparse second-order
sufficiency conditions.

\begin{defn}
   Let $x^*$ be an S-stationary point of \ref{SparseOpt}, with multipliers $(\lambda^*,\mu^*)$. Then
   we say that $(x^*,\lambda^*,\mu^*)$ satisfies
   \begin{enumerate}[label =(\roman*)]
      \item {\em SP-SOSC (sparse second-order sufficiency condition)} if
      \[
         d^T \nabla_{xx}^2 L^{SP}(x^*,\lambda^*,\mu^*) d > 0, \quad \forall d \in C^{SPO}(x^*,\lambda^*) \setminus \{ 0 \}  ,
      \]
      \item {\em strong SP-SOSC (strong sparse second-order sufficiency condition)} if
      \[
         d^T \nabla_{xx}^2 L^{SP}(x^*,\lambda^*,\mu^*) d > 0, \quad \forall d \in SC^{SPO}(x^*,\lambda^*) \setminus \{ 0 \}  .
      \]
   \end{enumerate}
\end{defn}

\noindent
Since the sparse critcial cone and sparse critical subspace are smaller than their 
standard counterparts, it follows that SP-SOSC and strong SP-SOSC are weaker assumptions
than standard SOSC and strong SOSC, respectively. We clarify the significance of SP-SOSC in the 
following result.

\begin{thm}[Second-Order Sufficiency Conditions]\label{Thm:SP-SOSC}
   Let $(x^*,\lambda^*,\mu^*)$ be an S-stationary point such that (strong) SP-SOSC holds in $ x^* $.
   Then the following statements hold:
   \begin{enumerate}[label = (\roman*)]
      \item (Strong) SOSC for \ref{PropPap} holds at $(x^*,y^*,\lambda^*,\mu^*,\gamma^*,\sigma^*)$ with $(y^*, \gamma^*, \sigma^*)$ defined in Proposition~\ref{EquivKKT}.
      \item (Strong) SOSC for \ref{ProbNew} holds at $(x^*,y^*,\lambda^*,\mu^*,\gamma^*)$ with $(y^*, \gamma^*)$ defined in Proposition~\ref{EquivKKT}.
      \item $x^*$ is a local minimizer of \ref{SparseOpt}.
   \end{enumerate}
\end{thm}

\begin{proof}
   For a given S-stationary point $(x^*,\lambda^*,\mu^*)$ let $y^*$, $\gamma^*$, and $\sigma^*$ be chosen as in Proposition~\ref{EquivKKT} and define $z := (x,y)$.
   The Hessian
   matrices of the Lagrangians of problems \ref{PropPap} and \ref{ProbNew} with respect to $z$ are 
   given by 
   \begin{align*} \nabla_{zz}^2 L^{lin} (x^*,y^*,\lambda^*,\mu^*,\gamma^*,\sigma^*)
      &= \begin{pmatrix} \nabla_{xx}^2 L^{SP}(x^*,\lambda^*,\mu^*) & \diag(\gamma^*) \\
      \diag(\gamma^*) & 0\end{pmatrix} \quad \text{and}\\ 
      \nabla_{zz}^2 L^{sq} (x^*,y^*,\lambda^*,\mu^*,\gamma^*)
      &= \begin{pmatrix} \nabla_{xx}^2 L^{SP}(x^*,\lambda^*,\mu^*) 
      & \diag(\gamma^*) \\
      \diag(\gamma^*) & \rho I_n \end{pmatrix},
   \end{align*}
   respectively, where $I_n$ denotes the identity matrix in $\RR^{n \times n}$. 
   Since $y_i^* = 1$ and $ \sigma_i^* = \rho > 0 $ for all $i \in I_0(x^*)$, we obtain the following
   critical cones for the smooth problems \ref{PropPap} and \ref{ProbNew}, respectively:
   \begin{align*}
      C^{lin}(z^*,\lambda^*) = \{ d = (d_x, d_y)^T \mid 
      \nabla g_i (x^*)^T d_x &= 0 \quad \forall i \in I_g(x^*),\ \lambda_i^* > 0, \\
      \nabla g_i(x^*)^T d_x &\le 0 \quad \forall i \in I_g(x^*),\ \lambda_i^* = 0, \\ 
      \nabla h(x^*)^T d_x &= 0, \\
      (d_x)_i &= 0 \quad \forall i \in I_0(x^*), \\
      d_y &= 0 \qquad \qquad \qquad \qquad  \qquad \}, \\
      C^{sq}(z^*,\lambda^*) = \{ d = (d_x, d_y)^T \mid 
      \nabla g_i (x^*)^T d_x &= 0 \quad \forall i \in I_g(x^*),\ \lambda_i^* > 0, \\
      \nabla g_i(x^*)^T d_x &\le 0 \quad \forall i \in I_g(x^*),\ \lambda_i^* = 0, \\ 
      \nabla h(x^*)^T d_x &= 0, \\
      (d_x)_i &= 0 \quad \forall i \in I_0(x^*), \\
      (d_y)_i &= 0 \quad \forall i \notin I_0(x^*) \qquad \qquad \}
   \end{align*}
   and, similarly, the critical subspaces
   \begin{align*}
      SC^{lin}(z^*,\lambda^*) := \{ d = (d_x, d_y)^T \mid
      \nabla g_i (x^*)^T d_x &= 0 \quad \forall i \in I_g(x^*),\ \lambda_i^* > 0, \\
      \nabla h(x^*)^T d_x &= 0, \\
      (d_x)_i &= 0 \quad \forall i \in I_0(x^*), \\
      d_y &= 0 \qquad \qquad \qquad \qquad  \qquad \}, \\
      SC^{sq}(z^*,\lambda^*) := \{ d = (d_x, d_y)^T \mid
      \nabla g_i (x^*)^T d_x &= 0 \quad \forall i \in I_g(x^*),\ \lambda_i^* > 0, \\
      \nabla h(x^*)^T d_x &= 0, \\
      (d_x)_i &= 0 \quad \forall i \in I_0(x^*), \\
      (d_y)_i &= 0 \quad \forall i \notin I_0(x^*) \qquad \qquad \}
   \end{align*}
   For a vector $ d = (d_x, d_y)^T $, we obtain
   \begin{eqnarray}
   \begin{pmatrix}d_x \\ d_y \end{pmatrix}^T \nabla_{zz}^2L^{sq}\begin{pmatrix}d_x \\ d_y\end{pmatrix} 
   & = & \begin{pmatrix}d_x \\ d_y \end{pmatrix}^T \nabla_{zz}^2L^{lin}\begin{pmatrix}d_x \\  
   d_y\end{pmatrix} + \rho \norm{d_y}_2^2 \nonumber \\
   & = & {d_x}^T \nabla_{xx}^2L^{SP}(x^*,\lambda^*,\mu^*) d_x + 2  (\gamma^*)^T (d_x \circ d_y) 
   + \rho \norm{d_y}_2^2. \label{AuxEq1}
   \end{eqnarray}
   Assume $d = (d_x,d_y)^T \in C^{lin}(x^*,\lambda^*)$ is a nonzero vector.
   Then we have
   \[
      d_x \in C^{SPO}(x^*,\lambda^*), \quad d_y = 0 .
   \]
   In particular, this implies $ d_x \neq 0 $.
   According to \eqref{AuxEq1}, the SP-SOSC immediately implies claim (i).
   The proof for strong SOSC is analogous.
   
   \par
   Assume $d = (d_x,d_y)^T \in C^{sq}(x^*,\lambda^*)$ is a nontrivial vector.
   It holds
   \[
      d_x \in C^{SPO}(x^*,\lambda^*), \quad (d_y)_i = 0, \ i \notin I_{0}(x^*).
   \]
   At least one of the two vectors $d_x,d_y$ is nonzero and we know $d_x \circ d_y = 0$.
   Hence SP-SOSC implies $(ii)$, according to inequality \eqref{AuxEq1}.
   Strong SOSC can again be verified analogously.

   Finally, the validity of SOSC for either \ref{PropPap} or \ref{ProbNew} 
   immediately yields $(iii)$ due to the equivalence of local minima.
\end{proof}

\noindent
We next state a second-order necessary optimality condition
for the sparse optimization problem \ref{SparseOpt}, which can be derived via the relation
to the corresponding second-order conditions of one of the two smooth reformulations
\ref{PropPap} or \ref{ProbNew}. Note that this necessary condition will not be used
later, but is stated here for the sake of completeness.

\begin{thm}[Second-Order Necessary Condition] 
   Let $x^*$ be a local minimum of \ref{SparseOpt} satisfying SP-LICQ. Then there exist unique 
   multipliers $(\lambda^*,\mu^*)$ such that $(x^*,\lambda^*,\mu^*)$ is an S-stationary point 
   of \ref{SparseOpt} satisfying the second-order necessary condition
   \[
      d^T  \nabla_{xx} L^{SP}(x^*,\lambda^*,\mu^*)  d \geq 0, 
      \quad \forall \ d \in C^{SPO}(x^*,\lambda^*).
   \]
\end{thm}

\begin{proof}
   The existence and uniqueness of the multipliers $ ( \lambda^*, \mu^* ) $ such that the 
   triple $ (x^*,\lambda^*,\mu^*) $ satisfies the S-stationarity conditions is an immediate
   consequence of Theorems~\ref{Thm:EquivStat} and \ref{EquivLICQ}.

   Furthermore, we know from these results that there exist (uniquely defined)
   vectors $ y^* $ and $ \sigma^* $ such that $(x^*,y^*,\lambda^*, \mu^*, \sigma^*)$ is a KKT
   point of \ref{ProbNew} satisfying standard LICQ, and with $(x^*,y^*)$
   being a local minimizer of \ref{ProbNew}, cf.\ Theorem~\ref{EquivLocMin}. Hence the
   standard second-order necessary optimality condition holds for \ref{ProbNew}, i.e., we have
   \[
      \begin{pmatrix}d_x \\ d_y \end{pmatrix}^T \begin{pmatrix} \nabla_{xx}^2 L^{SP}(x^*,\mu^*,\lambda^*) & \diag(\gamma^*) \\
      \diag(\gamma^*) & \rho I_n \end{pmatrix} \begin{pmatrix} d_x \\ d_y \end{pmatrix} \geq 0, \quad \forall 
      \begin{pmatrix}d_x \\ d_y \end{pmatrix} \in C^{sq}(x^*,\lambda^*).
   \]
   This is equivalent to
   \begin{equation} 
      (d_x)^T \ \nabla_{xx}^2 L^{SP}(x^*,\mu^*,\lambda^*)\ d_x + \rho \norm{d_y}_2^2 \geq 0, \quad 
      \forall 
      \begin{pmatrix}d_x \\ d_y \end{pmatrix} \in C^{sq}(x^*,\lambda^*),\label{Aux:eqSONC}
   \end{equation}
   where we used the fact that $(d_x)_i \cdot (d_y)_i = 0$, cf.\ the previous proof. Now it is easy to 
   see that any vector $d = (d_x, d_y)^T$ with
   $d_x \in C^{SPO}(x^*,\lambda^*)$ and $d_y = 0$ is contained in $C^{sq}(x^*,\lambda^*)$. In view of \eqref{Aux:eqSONC}, this directly yields
   \begin{equation*}
      (d_x)^T \ \nabla_{xx}^2 L^{SP}(x^*,\mu^*,\lambda^*)\ d_x \geq 0, \quad  \forall d_x \in C^{SPO}(x^*,\lambda^*).
   \end{equation*}
   This completes the proof.
\end{proof}
%
%

\noindent
Note that there exist more general second-order conditions for standard nonlinear
programs, see, e.g., \cite{PerturbationAnalysis}. In principle, it is possible to translate
these conditions also to problem-tailored second-order optimality conditions for the
sparse optimization problem \ref{SparseOpt} due to its relation to the standard
second-order optimality conditions to one of the reformulated smooth problems
\ref{PropPap} or \ref{ProbNew}. We omit the corresponding details.

\section{Lagrange-Newton-type Methods}\label{Sec:LagrangeNewton}

The aim of this section is to present some Lagrange-Newton-type methods
for the (local) solution of the sparse optimization problem \ref{SparseOpt}. The idea
is to use one of our smooth reformulations and to apply a Newton-type method to
the corresponding KKT conditions. In principle, we could take either the 
reformulation \ref{PropPap} or the one from \ref{ProbNew}. Here we decide to
consider the reformulation \ref{ProbNew} which, in particular, has the advantage
that the corresponding KKT conditions consist of nonlinear equations only if
the original problem \ref{SparseOpt} contains not inequalities. This observation
might be useful for Lagrange-Newton-type approaches. Nevertheless, the theory
also covers the case where inequality constraints are present.

More precisely, we  consider three different Newton-type methods: First, we 
take the full KKT system of \ref{ProbNew} and investigate the local convergence properties
of a corresponding nonsmooth Newton method applied to this system. Second, we 
consider a reduced variant of this method which eliminates the $ y $-variables and
show that it converges under the same set of assumptions as the previous approach.
Third, we deal with a method which tries to overcome some singularity problems
for some classes of sparse optimization problems which include nonnegativity constraints.

All three methods are using suitable {\em NCP-functions} $ \varphi: \RR^2 \to \RR $, which
are defined by the property
\begin{equation*}
   \phi (a,b) = 0 
   \qquad \Longleftrightarrow \qquad 
   a \geq 0, \;  b \geq 0, \; ab= 0 .
\end{equation*}
Two prominent examples are the {\em minimum function} and the {\em Fischer-Burmeister function}
\begin{equation*}
   \phi_m (a,b) := \min \{ a, b \}
   \qquad \text{and} \qquad
   \phi_{FB} (a,b) := \sqrt{a^2 + b^2} - a - b.
\end{equation*}
We need some background from nonsmooth analysis:
Given a locally Lipschitz continuous mapping $ T: \RR^n \to \RR^n $,
Rademacher's Theorem implies that $ T $ is almost everywhere differentiable. Hence the
set
\begin{equation*}
   \partial_B T(x) := \big\{ H \ \big| \ \exists \{ x^k \} \subseteq D_T: x^k \to x \text{ and }
   T'(x^k) \to H \big\} 
\end{equation*} 
is nonempty and bounded, where $ D_T $ denotes the set of differentiable points of $ T $.
The set $ \partial_B T(x) $ is called the {\em B-subdifferential} of $ T $ in $ x $, its
convex hull gives the {\em generalized Jacobian} $ \partial T(x) $ by Clarke \cite{Cla90}.
A point $ x $ is called {\em BD-regular}, if all elements in $ \partial_B T(x) $ are 
nonsingular. The nonsmooth Newton method 
\begin{equation*}
   x^{k+1} := x^k - H_k^{-1} \cdot T(x^k) \quad \forall k = 0, 1, 2, \ldots \quad \text{with} \quad H_k \in 
   \partial_B T(x^k)
\end{equation*}
for the solution of the nonlinear system of equations $ T(x) = 0 $ is known to be
superlinearly or quadratically convergent to a solution $ x^* $, if the solution is 
BD-regular and $ T $ satisfies an additional smoothness property called {\em semismoothness}
and {\em strong semismoothness}, respectively. For the precise definitions and proofs of
the previous statements, the interested reader is referred to the papers \cite{Qi93,QiS93}
and the monograph \cite{Facchinei}.

Throughout this section, we assume that all functions $ f, g, h $ are (at least)
twice continuously differentiable. Furthermore, $ \phi $ denotes either the minimum or
the Fischer-Burmeister function, unless we state something else explicitly.

The first Newton-type method presented in this section uses the operator
\[
   T(x,y,\lambda,\mu,\gamma) := \begin{pmatrix} 
      \nabla_x L^{SP}(x,\lambda,\mu) +  \gamma \circ y \\
      \rho(y-e) + \gamma \circ x \\
      \Phi_g(x,\lambda) \\
      h(x) \\
      x\circ y 
   \end{pmatrix} , 
\]
where $\Phi_g$ is defined componentwise by
\[
   (\Phi_g)_i(x,\lambda) = \phi(-g_i(x), \lambda_i) .
\]
Due to the defining property of an  NCP-function,
it follows that $ ( x^*, y^*, \lambda^*, \mu^*, \gamma^* ) $ is a KKT point of
the reformulated problem \ref{ProbNew} if and only if it solves the (in general nonsmooth)
system of equations $ T (x,y, \lambda, \mu, \gamma ) = 0 $. Furthermore, it is known that
the operator $ T $ is semismooth, and strongly semismooth if, in addition, the second-order
derivatives of $ f, g, h $ are locally Lipschitz continuous. In order to verify the 
local fast convergence of the corresponding nonsmooth Newton iteration
\[
   z^{k+1} := z^k - H_k^{-1} \cdot T(z^k), \quad \forall k = 0, 1, 2, \ldots,
\]  
with an arbitrary element $ H_k \in \partial_B T(z^k) $ and $ z^k: = (x^k,y^k,\lambda^k,\mu^k,
\gamma^k) $, it therefore suffices to verify the BD-regularity of a solution $ z^* $ of this system.
This is done in the following result.

\begin{thm}\label{TKKT}
   Let $z^* = (x^*,y^*,\lambda^*,\mu^*,\gamma^*)$ be a solution of $T(z)=0$ such that the following
   assumptions hold:
   \begin{enumerate}[label=(\roman*)]
      \item SP-LICQ is satisfied at $x^*$.
      \item Strong SP-SOSC is satisfied at $(x^*,\lambda^*,\mu^*)$.
   \end{enumerate}
   Then $ z^* $ is a BD-regular point of $ T $.
\end{thm}

\begin{proof}
   Based on our previous result, the statement can be traced back to existing results in
   the literature. Since $z^*$ is a KKT point of \ref{ProbNew}, we know 
   that the bi-active set $ \{ i \ | \ x_i^* = 0 = y_i^* \} $ is empty. Therefore, it 
   follows from assumption $ (i) $ and Theorem~\ref{EquivLICQ} 
   that ordinary LICQ holds for \ref{ProbNew} at $z^*$. Similarly, assumption $ (ii) $
   and Theorem~\ref{Thm:SP-SOSC} imply that the strong second-order sufficiency conditions holds 
   for \ref{ProbNew} at $z^*$. Standard results on the local convergence of nonsmooth
   Newton methods then imply that all elements $ H \in \partial_B T (z^*) $ are 
   nonsingular, see, e.g., \cite{FaK98,Facchinei,Fis92}.
\end{proof}

\noindent
We next consider a reduced formulation of the system $ T(z) = 0 $. To this end, note that
$ T(z) = 0 $ immediately gives
\begin{equation}\label{Eq:replacement}
   y = e - \frac{\gamma \circ x}{\rho} ,
\end{equation}
cf.\ \eqref{KKTnew}. Hence, eliminating the variable $ y $ in the definition of $ T $ by
replacing it using the above expression, we obtain the reduced operator
\[
   T_{red}(x,\lambda,\mu,\gamma) = \begin{pmatrix}
      \nabla_x L^{SP}(x,\lambda,\mu) + \gamma \circ (e - \frac{\gamma \circ x}{\rho}) \\
      \Phi(-g(x),\lambda)\\
      h(x) \\
      x \circ (e - \frac{\gamma \circ x}{\rho})
   \end{pmatrix},
\]
which is independent of $ y $. In view of its derivation, it still holds that any
zero of $T_{red}$ yields a KKT point of \ref{ProbNew} and vice versa, whenever the 
variable $y$ is defined as above. In order to locally solve the KKT system of 
\ref{ProbNew}, we can therefore, alternatively, apply a nonsmooth Newton method 
to the system $T_{red} (w) = 0 $, where $ w = (x, \lambda, \mu, \gamma) $. The central
point for the local fast convergence of this approach is the BD-regularity of a solution.
Here, the following result holds.

\begin{thm}\label{RedBD}
   $T$ is BD-regular in $(x,y,\lambda,\mu,\gamma)$ with $y = (e - \gamma \circ x /\rho)$ if and only if
   $T_{red}$ is BD-regular in $(x,\lambda,\mu,\gamma)$.
\end{thm}

\begin{proof}
   Let  $w = (x,\lambda,\mu,\gamma)$ and $z = (x,y,\lambda, \mu,\gamma)$ with 
   $y = e - \gamma \circ x / \rho $. The definition of the B-subdifferential then yields
   \[ 
      H \in \partial_B T(z) \quad \Longleftrightarrow \quad H = 
      \begin{pmatrix}
         \nabla_{xx}^2 L^{SP}(x,\lambda,\mu) & \diag(\gamma) & g'(x)^T &
         h'(x)^T & \diag(y) \\
         \diag(\gamma) & \rho I_n & 0 & 0 & \diag(x) \\
         J_1 \Phi_g & 0 & J_2 \Phi_g & 0 & 0 \\
         h'(x) & 0 & 0 & 0 & 0 \\
         \diag(y) & \diag(x) & 0 & 0 & 0
      \end{pmatrix} ,\]
   and, similarly, $ H_{red} \in \partial_B T_{red}(w)$ if and only if
   \begin{align*} 
      H_{red} =
      \begin{pmatrix} 
         \nabla_{xx}^2 L^{SP}(x,\lambda,\mu) - \frac{\diag(\gamma)^2}{\rho} & 
         g'(x)^T & h'(x)^T & 
         \diag(e - \frac{2 \gamma \circ x}{ \rho})  \\
         J_1 \Phi_g & J_2\Phi_g & 0 & 0\\
         h'(x) & 0 & 0 & 0 \\
         \diag(e - \frac{2 \gamma \circ x}{\rho}) & 0 & 0 & -\frac{\diag(x)^2}{\rho}
   \end{pmatrix} ,
   \end{align*}
   with $(J_1\Phi_g, J_2\Phi_g) \in \partial_B \Phi_g(x,\lambda)$. 
   Assume $ w $ is BD-regular for $T_{red}$. 
   Let $H \in \partial_B T(z)$ and consider the system
   \begin{equation} 
      H  d = 0 \quad \text{with appropriately partitioned} \quad 
      d = (d_x,d_y,d_{\lambda},d_{\mu},d_{\gamma}).
   \end{equation}
   Solving for $d_y$ explicitly and plugging in $y = e - \gamma \circ x / \rho$ yields
   \begin{align}
      \begin{split}
         \frac{1}{\rho} ( - \gamma \circ d_x - x\circ d_{\gamma}) - d_y & = 0, \\
         \begin{pmatrix}
            \nabla_{xx}^2 L^{SP}(x,\lambda,\mu) - \frac{\diag(\gamma)^2}{\rho} & 
            g'(x)^T & h'(x)^T & 
            \diag(e - \frac{2 \gamma \circ x}{ \rho}) \\
            J_1 \Phi_g & J_2\Phi_g & 0 & 0\\
            h'(x) & 0 & 0 & 0 \\
            \diag(e - \frac{2 \gamma \circ x}{\rho}) & 0 & 0 & -\frac{\diag(x)^2}{\rho} 
         \end{pmatrix} 
         \begin{pmatrix} d_x \\ d_{\lambda} \\ d_{\mu} \\ d_{\gamma} \end{pmatrix} & = 0 \ .
      \end{split}
   \end{align}
   BD-regularity of $T_{red}$ implies $(d_x,d_{\lambda},d_{\mu},d_{\gamma}) = (0,0,0,0)$ and therefore also
   $d_y = 0$. Hence $H$ is nonsingular. Since this holds for arbitrary $H 
   \in \partial_B T(z) $, the BD-regularity of $T$ in $z$ follows.

   The proof of the converse statement is similar: Assume $T_{red}$ is not BD-regular in $w$.
   Then there is a singular matrix $H^*_{red} \in \partial_B T_{red}(w^*)$, i.e., there exists
   $(J_1\Phi^*_g,J_2 \Phi^*_g) \in \partial_B \Phi_g(x^*,\lambda^*)$ such that 
   the corresponding element $H^*_{red}$ is singular.
   This means that there is a nontrivial element $ d_0 = (d_0^1, d_0^3, d_0^4, d_0^5)^T \in 
   \text{ker}(H^*_{red}) $. Setting $ d_0^2 := \frac{1}{\rho} ( - \gamma \circ d_0^1 - x\circ d_0^5) $
   and reversing the previous arguments, we obtain
   a singular element in $ \partial_B T(z) $.
\end{proof}

\noindent
Note that the assumption $ y = (e - \gamma \circ x /\rho) $ used in Theorem~\ref{RedBD}
holds automatically at any KKT point. Theorem~\ref{RedBD} therefore allows to
translate the result from Theorem~\ref{TKKT} directly to the reduced operator $T_{red}$.
A potential disadvantage of the reduced formulation is the fact that the replacement
of the variable $ y $ by the expression \eqref{Eq:replacement} increases the nonlinearity of
the resulting operator $ T_{red} $.

Finally, we turn to a third Newton-type method for the solution of sparse optimization
problems \ref{SparseOpt}, whose feasible set $X$ contains nonnegativity constraints for some or all variables.
For notational simplicity, we consider only the fully nonnegative case
\[ 
   x \ge 0 .
\]
In our general approach, we have to view these constraints as part of the inequalities $ g(x) \leq 0 $, which causes problems with the constraint qualification. SP-LICQ would require
the linear independence of the gradient vectors $ -e_i $ (resulting from the 
constraint $ x_i \geq 0 $ as an inequality) and $ e_i $ (resulting from the
sparsity in the definition of SP-LICQ) for all $ i \in I_0(x^*) $, which is obviously
impossible.

We can overcome this situation in the following way: In any local minimum of 
\ref{ProbNew}, we have $ y \geq 0 $ according to Lemma \ref{locmin}. 
Together with the constraint $ x \circ y = 0 $ and the nonnegativity constraint $ x \geq 0 $ we thus obtain the full complementarity conditions
$ x \geq 0, y \geq 0, x \circ y = 0 $, which we can replace by an NCP-function $\Phi(x,y) = 0$ with
$ \Phi_i(x,y) = \phi(x_i,y_i) $ for all $ i = 1, \ldots, n$.
The constraints $x \geq 0$ then do not need to be considered as a part of the standard
inequality constraints $ g(x) \leq 0 $ any more. This motivates to consider the 
nonlinear system of equations
\[ 
   T_C(x,y, \lambda, \mu, \gamma ) = 0 \quad \text{with} \quad
   T_C (x,y,\lambda,\mu,\gamma) := \begin{pmatrix}
      \nabla_x L^{SP}(x,\lambda,\mu) + \gamma \circ y \\
      \rho (y - 1) +  \gamma \circ x \\
      \Phi_g(x,\lambda)\\
      h(x) \\
      \Phi(x,y)
   \end{pmatrix} , 
\]
with two NCP-functions $\Phi_g, \Phi$. Then SP-LICQ is a reasonable assumption for this 
reformulation, and the following result holds.

\begin{thm}
   Let $z^*= (x^*,y^*,\lambda^*,\mu^*,\gamma^*)$ be a solution of $T_C(z)=0$ such that the assumptions
   of Theorem~\ref{TKKT} hold. Then $ z^* $ is a BD-regular point of $T_C$.
\end{thm}

\begin{proof}
   First observe that $ T_C(z^*) = 0 $ implies $ T (z^*) = 0 $, hence $ z^* $ is a KKT 
   point of \ref{ProbNew}. In view of Proposition~\ref{EquivKKT}, we therefore have that the 
   bi-active set $ \{ i \ | \ x_i^* = y_i^* = 0 \} $ is empty. This implies that $ \Phi $ 
   is continuously differentiable in a neighborhood of $ (x^*,y^*) $, with
   componentwise derivatives
   given by (recall that $\Phi $ is defined either by the Fischer-Burmeister
   function or by the minimum function)
   \begin{align*}
   \nabla \phi_{FB}(x_i^*, 0) &= (0, -1)^T \quad \text{and} \quad \nabla \phi_{m} (x_i^*, 0) = (0, 1)^T, \\
   \nabla \phi_{FB}(0, y_i^*) &=(-1, 0)^T \quad \text{and} \quad \nabla \phi_{m} (0, y_i^*) = (1, 0)^T.
   \end{align*}
   Thus, each element $H_C \in \partial_B T_C (z^*)$ can be written as:
   \[ 
      H_C = \begin{pmatrix}
         \nabla_{xx}^2 L^{SP}(x^*,\lambda^*,\mu^*) & 
         \diag(\gamma^*) & g'(x^*)^T & h'(x^*)^T & \diag(y^*) \\
         \diag(\gamma^*) & \rho I_n & 0 & 0 & \diag(x^*) \\
         J_1\Phi_g  & 0 & J_2\Phi_g & 0 & 0 \\
         h'(x^*) & 0 & 0 & 0 & 0 \\
         \diag(c_x) & \diag(c_y) & 0 & 0 & 0 
      \end{pmatrix}, 
   \]
   with $c_x, c_y$ such that
   \[ 
      ((c_x)_i,(c_y)_i) \in \begin{cases}
         \{-1, 1\} \times \{0\} \quad \text{if} \ i \in I_0(x^*), \\
         \{0\} \times \{-1,1\} \quad \text{otherwise},
      \end{cases} 
   \] 
   and arbitrary
   $(J_1 \Phi_g, J_2\Phi_g) \in \partial_B \Phi_g (x^*,\lambda^*)$.
   Define
   \[ 
      A := \begin{pmatrix}
         I_{2n + m + p} & 0  \\
         0 & \diag \big( (c_x + c_y) \circ (x^* + y^*) \big)
      \end{pmatrix}, 
   \]
   and observe that $ A $ is nonsingular. Then a simple calculation shows that 
   $ A \cdot H_C \in \partial_B T (z^*) $. Since $ A $ is nonsingular and all elements in
   $ \partial_B T(z^*) $ are nonsingular by Theorem~\ref{TKKT}, it follows that $ H_C $ is nonsingular.
   This completes the proof.
\end{proof}

\noindent
Though the third formulation using the operator $ T_C $ is mainly designed for
problems having additional nonnegativity constraints,  we can also
apply this idea also to problems without these nonnegativity constraints, 
by splitting the variables $x = x^+ - x^-$ into their positive and negative parts $x^+ x^- \geq 0$.
Since this is  a pretty standard approach also used in \cite{l0}, we skip the corresponding details.

We close with a comment regarding the choice of the NCP-function. From a purely
local point of view, the previous considerations indicate that there is, basically,
no difference between using the Fischer-Burmeister  or the minimum function.
Nevertheless, in our subsequent implementation, we prefer to use the Fischer-Burmeister
approach simply because the (generalized) partial derivatives of the minimum function have the 0-1-entries, whereas the corresponding partial derivatives of
the Fischer-Burmeister-function are usually both different from zero (unless we are in
a KKT point). This implies, in a sense, that it is more likely to generate 
singular Jacobians for the minimum-function than for the 
Fischer-Burmeister function.

\section{Numerical Results}\label{Sec:Numerics}

In this section we present some numerical results obtained by applying the previously developed Lagrange-Newton-type methods to some commonly known fields of sparse optimization problems. We start with some preliminaries regarding our implementation.

\subsection{Implementation}

\paragraph{Initial Values}

Lagrange-Newton-type methods are mainly locally
convergent approaches. Our aim is to show
these methods can be used to improve solutions obtained
by globally convergent techniques. Therefore, we pre-process the problem by first solving the $\ell_1$-surrogate problem
\[
   \min_x \ f(x) + \rho \norm{x}_1 \ST x\in X,
\]
with $f, X$ as in \ref{SparseOpt}. 
We then use the solution $x_{\ell_1}$ of the $\ell_1$-surrogate problem as initial point $x^0$ for the Lagrange-Newton-type methods, which we consider post-processing of the $\ell_1$-surrogate problem.
Accumulation points $x^*$ of our Lagrange-Newton-type methods should (hopefully) be preferable for \ref{SparseOpt} over the $\ell_1$-solution.

Note that it is, in general, not useful to have $x^0 = 0$ as the initial guess. In fact,
in cases where constraints do not exist, the initial guess $x^0=0$ does already yield an S-stationary point. The starting 
point $ x^0 = x_{\ell_1} $, obtained by the pre-preprocessing
phase, may also have many zero components, but should, nonetheless, be
a much better choice than the zero vector.
Furthermore, we found it beneficial to initialize $y^0 := e$ since we want to see a majority of $0$-entries in the accumulation point $x^*$ of the algorithm, which would correlate with a $y^*$ consisting of mainly $1$-entries. For any of the Lagrangian-multipliers $(\lambda, \mu, \gamma)$ we agreed on the canonical choice: $\lambda^0 =0$, $\mu^0 = 0$, $\gamma^0 = 0$, in the respective dimensions. Note that any choice of $\gamma^0$ might be arbitrarily bad since, for an accumulation point $x^*$ with an entry $10^{-4} \approx |x^*_i| \neq 0$, one has to expect $\gamma_i^* \approx \rho\ \text{sign}(x^*)10^4$.

\paragraph{Dealing with the B-Subdifferential}

We only consider the Fischer-Burmeister function, whenever an NCP-function is required in our computations. The method to obtain an element in the B-subdifferential of the Fischer-Burmeister function is widely known, compare~\cite{KanzowFacchinei}. We fix a point $z = (x,y,\lambda,\mu,\gamma)$ and consider the operator $T_C$ with the components:
\[
   \phi_{FB}(x_i,y_i), \ (i=1,...,n), \quad \phi_{FB}(-g_j(x), \lambda_j), \ (j = 1,...,m),
\]
and
\begin{equation*}
   I^{xy} :=\{ i \ | \ x_i = y_i = 0\}, \quad
   I^{g\lambda} :=\{ j \ | \ g_j(x) = \lambda_j = 0\}.
\end{equation*}
Define:
\[
   (x^t, y^t, \lambda^t):= (x-te(n),y-te(n),\lambda-te(p)), \quad \text{for}\ t>0,
\]
with $e = (1,1,...,1)^T$ of the appropriate dimension. Passing to the limit $t\searrow 0$ yields
\begin{equation}\label{BDiffxy}
   \lim_{t\searrow 0} \ \nabla_{(x_i,y_i)} \phi_{FB}(x_i^t,y_i^t)^T = \begin{cases} 
      \Big( \frac{x_i}{\sqrt{x_i^2 + y_i^2}} - 1,\ \frac{y_i}{\sqrt{x_i^2 + y_i^2}}-1 \Big), \quad &i \notin I^{xy},\\ 
      \Big( -\frac{1}{\sqrt{2}} - 1,\ -\frac{1}{\sqrt{2}}-1 \Big), \quad &i \in I^{xy},
   \end{cases}
\end{equation}
and by applying the mean-value theorem to $g_j$, we further have
\begin{align}
   \lim_{t\searrow 0} \ \nabla_{(x_j,\lambda_j)} &\phi_{FB}(-g_j(x^t),\lambda_j^t)^T=\notag\\
   \label{BDiffugl}
   &=\begin{cases}
      \Big( \Big( \frac{g_j(x)}{\sqrt{g_j(x)^2 + \lambda_j^2}} + 1 \Big) g_j'(x),\ \frac{\lambda_j}{\sqrt{g_j(x)^2 + \lambda_j^2}} - 1 \Big), \quad &j \notin I^{g\lambda},\\
      \Big( \Big( -\frac{\nabla g_j(x)^T e}{\sqrt{(\nabla g(x)^Te)^2 + 1}} + 1 \Big) g_j'(x),\ -\frac{1}{\sqrt{(\nabla g(x)^Te)^2 + 1}} - 1 \Big), \quad &j \in I^{g\lambda},
   \end{cases}
\end{align}
which are elements of the B-subdifferential of the Fischer-Burmeister function.

\paragraph{Termination Criterion}

The canonical condition for terminating one of our Newton-type methods with operator $T$ would be
\[
   \norm{T(x)} \le \varepsilon,
\]
with some sufficiently small tolerance $\varepsilon$.
Unfortunately we occasionally observe the problematic behavior that in some components $x^k_i \rightarrow 0$, but at the same time $y_i^k \rightarrow 0$ and $\gamma_i^k \rightarrow \infty$.
Recall that at a minimum or stationary point $(x^*,y^*)$, we should instead have $y^*_i = 1$ for all $i$ with $x^*_i = 0$.
When we observe the behavior, typically the iterates $x^k$ are nonetheless sufficiently feasible and the gradient of the Lagrangian to $L^{SP}$ is sufficiently small in every component $i$ with $x_i^k \not \approx 0$, which points to the fact that the accumulation point is S-stationary. 
We therefore terminate the algorithms, when the following check for S-stationarity is satisfied:

\begin{enumerate}[label = (S.\arabic*)]
   \item Choose tolerances $\delta\ge 0, \varepsilon \ge 0$ and define the set of nonzero components as
   \[
      I_{\neq 0}:=\{i\ |\ |x_i^k| \ge \delta \}.
   \]
   \item Set $ L := \nabla_x L^{SP}(x^k,\mu^k,\lambda^k) = \nabla f(x^k) + g'(x^k)^T \lambda^k + h'(x^k)^T \mu^k  $ and compute
   \[
       \text{res} = \norm { \begin{pmatrix} L_{I^{\neq 0}} \\ \Phi_g(x^k,\lambda^k) \\ h(x^k) \end{pmatrix}} 
       \qquad \qquad
       \left( \text{or} \quad \text{res} = \norm { \begin{pmatrix} L_{I^{\neq 0}}  \\ \Phi_g(x^k,\lambda^k) \\ h(x^k) \\ \max\{ 0, -x^k_{I^{\neq 0}}\}\end{pmatrix}} \quad \text{in case } x\ge 0 \right). 
   \]
   \item Terminate the iteration, if $\text{res} \le \varepsilon$.
\end{enumerate}

\noindent
In our application, we set $\delta = 10^{-4}$ and $\varepsilon = 10^{-6}$.

\subsection{Sparse Portfolio Selection}

The portfolio optimization problem in the sense of Markowitz \cite{Markowitz1952} can be represented as
\begin{equation}\label{Pf}
   \min_x \ \frac{1}{2} x^T Q x \ST e^T x = 1, \; \alpha^T x \ge \beta, \; x\ge 0, 
\end{equation}
where $x_i$ denotes the amount of asset $i$ bought, $\alpha_i$ is the expected payout of asset $i$ and $Q \in \RR^{n \times n}$ is the covariance matrix of all payouts.
If additionally an investor is interested in having only a few active assets, this can be formulated as a sparse optimization problem
\begin{equation}\label{spPf}
   \min_x \ \frac{1}{2} x^T Q x + \rho \norm{x}_0 \ST e^T x = 1, \; \alpha^T x \ge \beta, \; x\ge 0 .
\end{equation}
Pre-processing this sparse problem with the $\ell_1$-norm does not yield any useful result, because for all $x \ge 0$ we have $ \norm{x}_1 = e^T x $, which is constant on the feasible set of (\ref{spPf}).
Therefore, solutions of the $\ell_1$-surrogate problem for (\ref{spPf}) coincide with solutions of (\ref{Pf}).
We thus solve (\ref{Pf}) in our numerical tests to obtain an initial point $x^0$ minimizing $x^T Q x$ on the feasible set and then use Lagrange-Newton-type methods to search for a sparse value $x^*$ in its vicinity.

We ran our tests in MATLAB\footnote{\url{https://de.mathworks.com/products/matlab.html}} R2020b and used the set of portfolio selection test problems from Frangioni and Gentile\footnote{\url{http://groups.di.unipi.it/optimize/Data/MV.html}}. Note that in order to obtain the form (\ref{spPf}), we neglected the upper and lower bounds on entries $i\notin I_0(x)$ of $x$.
The initial point $x^0$ was obtained by applying the \texttt{quadprog}-function of MATLAB to problem (\ref{Pf}).
In the Lagrange-Newton-type methods, the restriction $x \ge 0$ was then only explicitly incorporated in the operator $T_C$.
For $T$ and $T_{red}$ these sign constraints were only considered in the termination criterion, but not present in the Lagrange-Newton-step. 
Nonetheless, for all test instances and all operators $T, T_{red}, T_C$ the algorithms terminated within $100$ steps in an $\varepsilon$-feasible point.

The goals was to iterate from the $x^0$ to a point, which is still sufficiently good with regards to the objective $f(x)$, but is of much higher sparsity than $x^0$.
For $\rho = 1$ and dimension $n=400$, the resulting values of $f(x) + \norm{x}_0$ for the initial value $x^0$ and the three Lagrange-Newton-type methods are given in Figure~\ref{fig:portfolio}.
The average amount of necessary iterations for each of the three methods was:
\begin{center}
   \begin{tabular}{l | l | l | l}
      & $T_C$ & $T_{RED}$ & $T$ \\
      \hline
      avg. numb. of iter. & 12,9 & 45,6 & 36,3\\
   \end{tabular}
\end{center}

\begin{figure}[htbp]
   \center
   \includegraphics[height = 0.2\textheight]{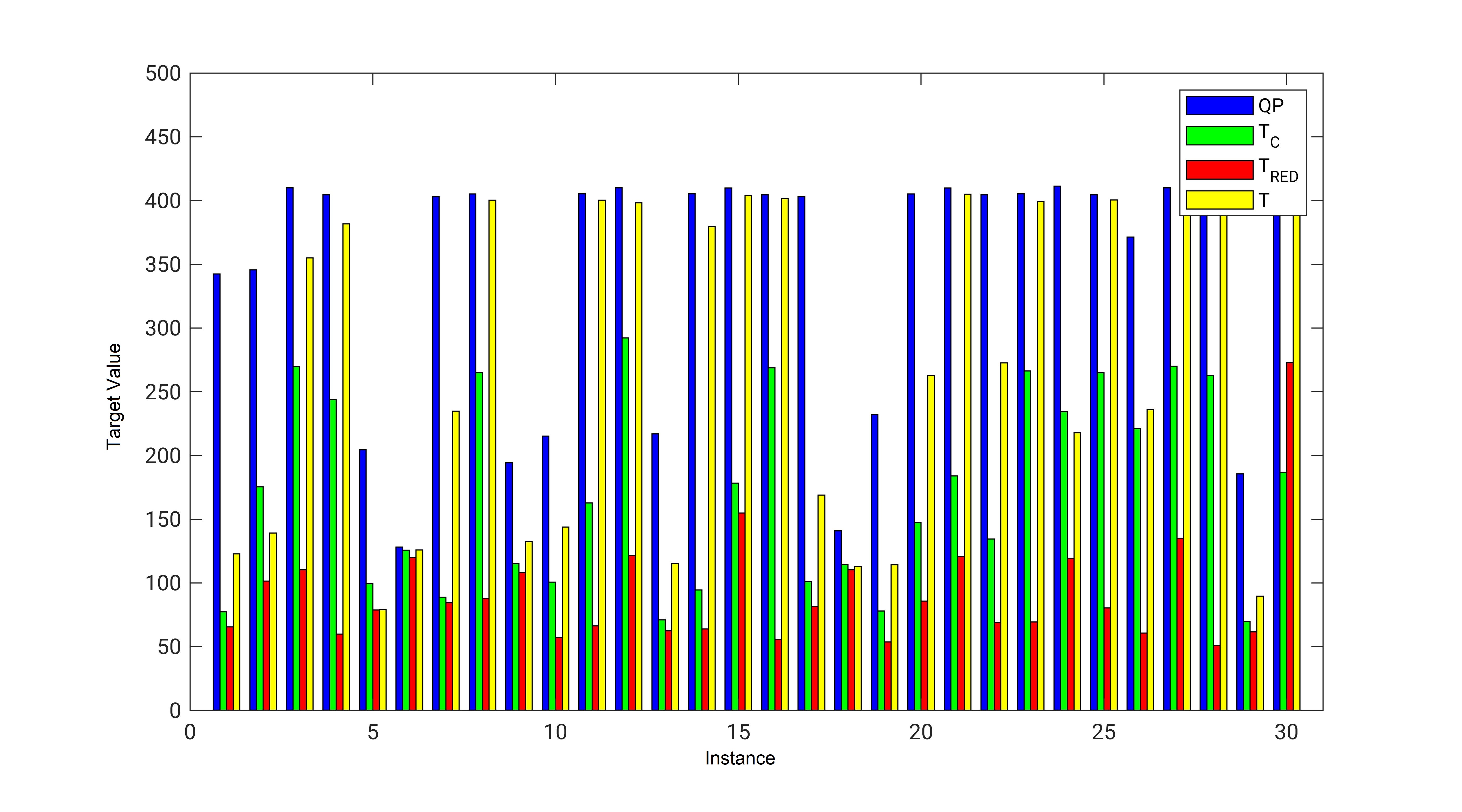}
   \caption{Target value $f(x) + \norm{x}_0$ for portfolio selection with dimension $n=400$}\label{fig:portfolio}
\end{figure}

In every instance we were able to improve on the solution QP, given by the $quadprog$-approach, with any of the solutions obtained by $T, T_{red}$ and $T_C$.
Almost always, $T_{red}$ led to the best results followed by $T_C$ and finally $T$.
However,  $T_{red}$ as well as $T$ have a much higher iteration count than anticipated for a Newton-type method, which could be caused by the difficult structure of the constraints
$ x \circ y = 0 $ and $ x \circ ( 1 - \gamma \circ x /\rho) = 0 $,
and the lack of a good initial guess $x^0$.
Only $T_C$, where complementarity between $x$ and $y$ was handled by $\phi_{FB}$, delivered a sufficiently low iteration count.

\subsection{Compressive Sensing}

In its essence, compressive sensing deals with reconstructing an $n$-dimensional vector $\overline{x}$ encoded by some sensing-matrix $A \in \RR^{m \times n}$ with $m \ll n$ into a signal $A\overline{x} = b$ of much lower dimension.
Assuming that the original signal $\overline{x}$ was sparse sparse leads to the following formulation for compressive sensing
\begin{equation}\label{CS}
   \min_x \ \norm{x}_0 \ST Ax = b, 
\end{equation}
which was studied by Tao and Candès \cite{1542412}.
Problem (\ref{CS}) can be seen as an instance of \ref{SparseOpt} with $f\equiv 0$.
Since we need some second order information for our local Newton-type methods, instead of the noise-free problem (\ref{CS}) we are more interested in the problem
\begin{equation*} 
   \min_x \ \norm{x}_0 \ST \norm{Ax - \overline{b}}_2 \le \delta,
\end{equation*}
with some tolerance $\delta > 0$.
This problem is motivated by the assumption that the received signal is $\overline{b} = b + r$ with some noise $r$.
However, this formulation requires that the noise level $\delta$ is known at least approximately.
To avoid this problem, a penalty formulation
\begin{equation}\label{CSN-penalty}
   \min_x \ \frac{1}{2} \norm{Ax - \overline{b}}^2_2 + \rho \norm{x}_0 
\end{equation}
as seen in \cite{Wang2019} is often considered instead.
Replacing the $\ell_0$-norm with the convex, sparsity inducing $\ell_1$-norm results in the \textit{basis persuit denoising} problem, presented as in \cite{Chen1998}:
\begin{equation}\label{BPD}
   \min_x \ \frac{1}{2} \norm{Ax - \overline{b}}^2_2 + \rho \norm{x}_1 .
\end{equation}
In our numerical test, we compute an initial point $x^0$ by solving the $\ell_1$-surrogate problem (\ref{BPD}) and then use $x^0$ together with the three Newton-type methods to solve (\ref{CSN-penalty}).

We set up our examples as in \cite{zhao2021lagrangenewton}:
Let $SA \in \RR^{(m+p)\times n}$ be some sensing-matrix and $\overline{x} \in \RR^n$ be some sparse vector. We set:
\[
   Sb:= SA \cdot \overline{x},
   \]
and split $SA$ and $Sb$ into:
\[
   SA = \begin{pmatrix} A \\ C \end{pmatrix}, \quad A \in \RR^{m\times n}, C \in \RR^{p \times n}, \quad 
   Sb = \begin{pmatrix} b \\ d \end{pmatrix}, \quad b \in \RR^m, d \in \RR^p.
\]
We then consider the following problem
\begin{equation}\label{CSPO}
  \min_x \ F_{\rho}(x) := \frac{1}{2} \norm{Ax - b}_2^2 + \rho \norm{x}_0 \ST Cx = d, 
\end{equation}
where the linear constraints $Cx = d$ can be considered as noise-free information and exclude $x = 0$ from the feasible set.
The dimensions were set to $n = 512, m = 128$ and $p = 8$, the sparsity of $\overline{x}$ was chosen as $s = \norm{\overline{x}}_0 = 32$. The sensing-matrix $SA$ was initialized as a Gauß-matrix as in \cite{Yin2015}, such that:
\[
   SA_j \sim \mathcal{N}(0, E_{p+m} / (p+m)), \quad \forall j =1, ..., n .
\]
Closely following \cite{zhao2021lagrangenewton}, we initialized the components of (\ref{CSPO}) as
\begin{align*}
   &\bar{x} = \texttt{zeros}(n,1),  && \Gamma = \texttt{randperm}(n),  && \bar{x}(\Gamma(1:s)) = \texttt{randn}(s,1), && Sb = SA \bar{x}, \\
   & J = \texttt{randperm}(p+m), && J_1 = J(1:m), && J_2 = J(m+1:\textit{end}), && \\
   & A = SA(J_1), && b = Sb(J_1), && C=SA(J_2), && d = Sb(J_2) .
\end{align*}
To obtain an initial guess $x^0$, we considered the $\ell_1$-problem \eqref{BPD} as a quadratic program and applied MATLAB's \texttt{quadprog} as seen in \cite{Gaines2018}, which required the split $x = x_+ - x_-$ with $x_+,x_- \geq 0$.
From the solution $(x_+^0, x_-^0)$ we could recover $x^0 = x_+^0 - x_-^0$.
Note that in order to invoke the operator $T_C$, we now have to split $x$ into positive and negative part, because otherwise we do not have any nonnegativity constraints.
For $T_C$ we thus used $(x_+^0, x_-^0)$ as is to initialize the algorithm.
Unfortunately, this split leads to a much higher computational cost for $T_C$, since in every Newton-step a system of $6n + m$ equations had to be solved.

This time we chose a discrete set $\{0.1, 0.5, 1, 2, 3, 4, 5\}$ of values for $\rho$ and ran $100$ test examples for each of those values. We were faced with some unsuccessful runs regarding $T_C$, where the algorithm failed to converge in 5.6\% of all tests, since either the iteration number exceeded the maximum of $100$ steps or we had to terminate early, as the error in the Newton-step with respect to the $\ell_2$-norm went past the safety threshold of $100$.
For all values of $\rho$, the resulting average value $f(x) + \rho \norm{x}_0$ of all successful runs is shown in Figure~\ref{fig:sensing}.
Again, we observe a significant improvement of the objective function value for all operators $T, T_{red}, T_C$, but now with less pronounced differences between the three operators.

\begin{figure}[htbp]
   \centering
   \includegraphics[height = 0.2\textheight]{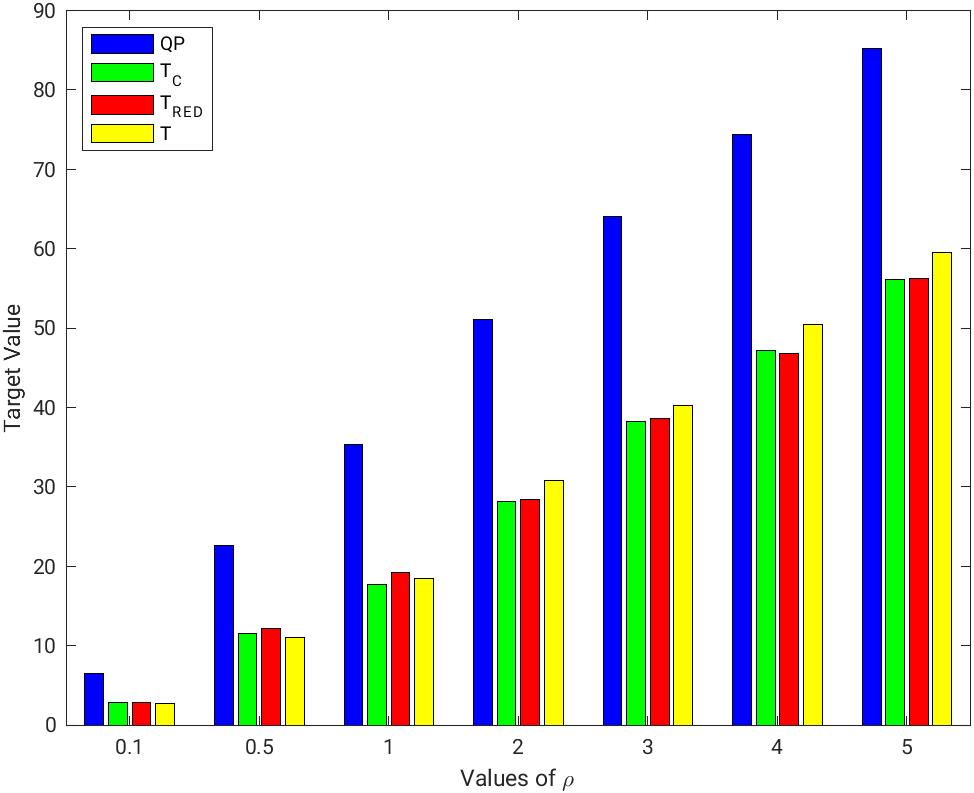}
   \caption{Average target value of $f(x) + \rho \norm{x}_0$ for successful compressive sensing runs}\label{fig:sensing}
\end{figure}

\subsection{Logistic Regression}

Consider the following sparse optimization problem
\begin{equation}\label{logit}
   \min_w \  \sum_{i=1}^m \log(1 + \exp(-y_i \cdot w^T x_i)) + \rho \norm{w}_0,
\end{equation}
which we refer to as the penalized maximum log-likelihood function.
This estimator is applied to match a sigmoid-function to a set of measurements $x_1,...,x_m$ and corresponding Bernoulli-variables $y_1,...,y_n \in \{-1,1\}^m$, where additionally sparsity is promoted in the parameters $w_i$.
Replacing $\norm{\cdot}_0$ by $\norm{\cdot}_1$ in (\ref{logit}), we obtain a  convex composite optimization problem, which can be tackled by FISTA or proximal BFGS methods, compare \cite{2014}.

In our numerical test, we consider the problem \textit{gisette} from the NIPS 2003 feature selection challenge, which
was acquired from the \textit{LIBSVM}-website\footnote{\url{https://www.csie.ntu.edu.tw/~cjlin/libsvmtools/datasets/}}.
The classification problem is high-dimensional $(n = 5000, m = 6000)$ and was scaled to $[-1,1]$.
Recall that applying either of the Newton-type methods with $T_C, T$ or $T_{red}$ to the \textit{gisette} problem leads to a drastic increase in the dimensionality (in the case of $T_C$: $n = 30000$).
Computation was therefore outsourced to a faster PC and handled in \textit{Python}.

We computed an initial point $x^0$ by solving the $\ell_1$-surrogate problem to (37) with FISTA. Running the three Newton-type methods with this initial point then lead to the results in Figure~\ref{fig:regression}. As one can see, all three of the operators lead to an improved sparsity $\|x\|_0$ and an improved function value $f(x) + \|x\|_0$, meaning a better solution of the original problem (\ref{logit}).

\begin{figure}[htbp]
   \centering
   \includegraphics[height = 0.2\textheight]{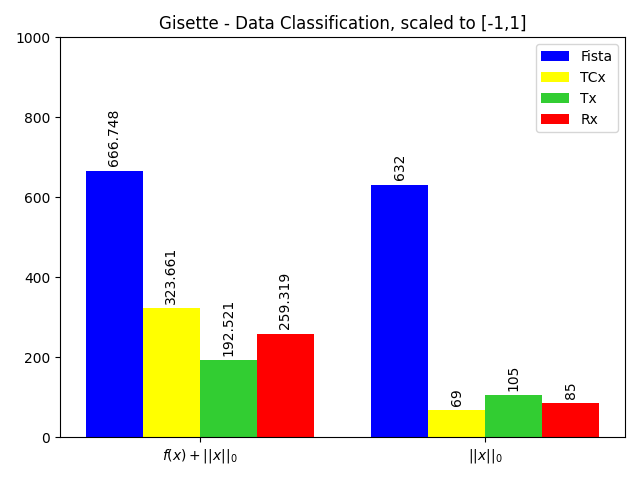}
   \caption{Comparison of target value $f(x) + \|x\|_0$ and sparsity $\|x\|_0$ for logistic regression}\label{fig:regression}
\end{figure}

\section{Final Remarks}\label{Sec:Final}

The aim of this paper was mainly to lay the theoretical foundation fpr two reformulations of the highly difficult sparse optimization problem SPO.
In particular, it was shown that we get full equivalence of problem SPO with these
two reformulations in terms of global and local minima.
Moreover, the corresponding stationary conditions also
coincide and corresponding second-order conditions are closely related.
These results can be used to develop and investigate Lagrange-Newton-type methods for the numerical solution of problem SPO and the numerical results indicate that one can use these methods in order to get significant improvements of solutions obtained by some other techniques.

The Lagrange-Newton-type methods, of course, are local
in nature, but result quite naturally as a direct 
consequence of our theoretical considerations. Our
future research, however, will concentrate on 
the development of globally convergent methods based
on our reformulations. Some preliminary results in
this direction can already be found in \cite{PhD-ThesisRaharja}.


\bibliographystyle{abbrv} 
\bibliography{Quellen} 

\begin{thebibliography}{10}

\bibitem{Beck2017}
A.~Beck.
\newblock {\em First-Order Methods in Optimization}.
\newblock SIAM, 2017.

\bibitem{ben2021global}
R.~Ben~Mhenni, S.~Bourguignon, and J.~Ninin.
\newblock Global optimization for sparse solution of least squares problems.
\newblock {\em Optimization Methods and Software}, pages 1--30, 2021.

\bibitem{bertsimas2020sparse}
D.~Bertsimas and B.~Van~Parys.
\newblock Sparse high-dimensional regression: Exact scalable algorithms and
  phase transitions.
\newblock {\em The Annals of Statistics}, 48(1):300--323, 2020.

\bibitem{PerturbationAnalysis}
J.~F. Bonnans and A.~Shapiro.
\newblock {\em Perturbation Analysis of Optimization Problems}.
\newblock Springer New York, 2000.

\bibitem{CCSPO}
O.~P. Burdakov, C.~Kanzow, and A.~Schwartz.
\newblock Mathematical programs with cardinality constraints: reformulation by
  complementarity-type conditions and a regularization method.
\newblock {\em SIAM Journal on Optimization}, 26(1):397--425, 2016.

\bibitem{1542412}
E.~Candes and T.~Tao.
\newblock Decoding by linear programming.
\newblock {\em IEEE Transactions on Information Theory}, 51(12):4203--4215,
  2005.

\bibitem{CKS16}
M.~{\v{C}}ervinka, C.~Kanzow, and A.~Schwartz.
\newblock Constraint qualifications and optimality conditions for optimization
  problems with cardinality constraints.
\newblock {\em Mathematical Programming}, 160(1):353--377, 2016.

\bibitem{Chen1998}
S.~S. Chen, D.~L. Donoho, and M.~A. Saunders.
\newblock Atomic decomposition by basis pursuit.
\newblock {\em {SIAM} Journal on Scientific Computing}, 20(1):33--61, 1998.

\bibitem{Chen2017}
X.~Chen, L.~Guo, Z.~Lu, and J.~J. Ye.
\newblock An augmented {L}agrangian method for non-{L}ipschitz nonconvex
  programming.
\newblock {\em SIAM Journal on Numerical Analysis}, 55(1):168--193, 2017.

\bibitem{Cla90}
F.~H. Clarke.
\newblock {\em Optimization and {N}onsmooth {A}nalysis}.
\newblock SIAM, 1990.

\bibitem{KanzowFacchinei}
T.~De~Luca, F.~Facchinei, and C.~Kanzow.
\newblock A semismooth equation approach to the solution of nonlinear
  complementarity problems.
\newblock {\em Mathematical Programming}, 75:407--439, 1996.

\bibitem{DeMarchi2022}
A.~De~Marchi, X.~Jia, C.~Kanzow, and P.~Mehlitz.
\newblock Constrained structured optimization and augmented {L}agrangian
  proximal methods.
\newblock Technical report, Institute of Mathematics, University of Würzburg,
  April 2022.

\bibitem{FaK98}
F.~Facchinei, A.~Fischer, and C.~Kanzow.
\newblock Regularity properties of a semismooth reformulation of variational
  inequalities.
\newblock {\em SIAM Journal on Optimization}, 8(3):850--869, 1998.

\bibitem{Facchinei}
F.~Facchinei and J.-S. Pang.
\newblock {\em Finite-Dimensional Variational Inequalities and Complementarity
  Problems}.
\newblock Springer New York, 2004.

\bibitem{l0}
M.~Feng, J.~E. Mitchell, J.-S. Pang, X.~Shen, and A.~Wächter.
\newblock Complementarity formulations of $\ell_0$-norm optimization problems.
\newblock {\em Pacific Journal of Optimization}, 14(2):273 -- 305, 2018.

\bibitem{Fis92}
A.~Fischer.
\newblock A special {N}ewton-type optimization method.
\newblock {\em Optimization}, 24(3-4):269--284, 1992.

\bibitem{Gaines2018}
B.~R. Gaines, J.~Kim, and H.~Zhou.
\newblock Algorithms for fitting the constrained lasso.
\newblock {\em Journal of Computational and Graphical Statistics},
  27(4):861--871, Aug. 2018.

\bibitem{Ghilli2017}
D.~Ghilli and K.~Kunisch.
\newblock A monotone scheme for sparsity optimization in $ \ell_p $ with $ p
  \in (0, 1] $.
\newblock {\em IFAC-PapersOnLine}, 50(1):494--499, 2017.

\bibitem{Gotoh-Takeda-Tono-2018}
J.-y. Gotoh, A.~Takeda, and K.~Tono.
\newblock {DC} formulations and algorithms for sparse optimization problems.
\newblock {\em Mathematical Programming}, 169(1):141--176, 2018.

\bibitem{HoK08}
T.~Hoheisel and C.~Kanzow.
\newblock Stationary conditions for mathematical programs with vanishing
  constraints using weak constraint qualifications.
\newblock {\em Journal of Mathematical Analysis and Applications},
  337(1):292--310, 2008.

\bibitem{HKS16}
T.~Hoheisel, C.~Kanzow, and A.~Schwartz.
\newblock Theoretical and numerical comparison of relaxation methods for
  mathematical programs with complementarity constraints.
\newblock {\em Mathematical Programming}, 137(1):257--288, 2013.

\bibitem{LeThi-PhamDinh-Le-Vo-2014}
H.~A. Le~Thi, T.~Pham~Dinh, H.~M. Le, and X.~T. Vo.
\newblock {DC} approximation approaches for sparse optimization.
\newblock Technical report, 2014.

\bibitem{2014}
J.~D. Lee, Y.~Sun, and M.~A. Saunders.
\newblock Proximal {N}ewton-type methods for minimizing composite functions.
\newblock {\em SIAM Journal on Optimization}, 24(3):1420--1443, jan 2014.

\bibitem{Lu-Zhang-2012}
Z.~Lu and Y.~Zhang.
\newblock Penalty decomposition methods for $l_0$-norm minimization.
\newblock Technical report, 2012.

\bibitem{Markowitz1952}
H.~Markowitz.
\newblock Portfolio selection.
\newblock {\em The Journal of Finance}, 7(1):77--91, Mar. 1952.

\bibitem{Meh20}
P.~Mehlitz.
\newblock Stationarity conditions and constraint qualifications for
  mathematical programs with switching constraints.
\newblock {\em Mathematical Programming}, 181(1):149--186, 2020.

\bibitem{Qi93}
L.~Qi.
\newblock Convergence analysis of some algorithms for solving nonsmooth
  equations.
\newblock {\em Mathematics of Operations Research}, 18(1):227--244, 1993.

\bibitem{QiS93}
L.~Qi and J.~Sun.
\newblock A nonsmooth version of {N}ewton's method.
\newblock {\em Mathematical programming}, 58(1):353--367, 1993.

\bibitem{PhD-ThesisRaharja}
A.~B. Raharja.
\newblock {\em Optimisation Problems with Sparsity Terms: Theory and
  Algorithms}.
\newblock Phd {T}hesis, Julius-Maximilians-Universität Würzburg, 2020.

\bibitem{Tillmann-et-al-2021}
A.~M. Tillmann, D.~Bienstock, A.~Lodi, and A.~Schwartz.
\newblock Cardinality minimization, constraints, and regularization: A survey.
\newblock Technical report, 2021.

\bibitem{Wang2019}
L.~Wang, J.~Wang, J.~Xiang, and H.~Yue.
\newblock A re-weighted smoothed $\ell_0$-norm regularized sparse reconstructed
  algorithm for linear inverse problems.
\newblock {\em Journal of Physics Communications}, 3(7):075004, July 2019.

\bibitem{Yin2015}
P.~Yin, Y.~Lou, Q.~He, and J.~Xin.
\newblock Minimization of $\ell_{1-2}$ for compressed sensing.
\newblock {\em {SIAM} Journal on Scientific Computing}, 37(1):A536--A563, Jan.
  2015.

\bibitem{zhao2021lagrangenewton}
C.~Zhao, N.~Xiu, H.~Qi, and Z.~Luo.
\newblock A {L}agrange-{N}ewton algorithm for sparse nonlinear programming.
\newblock {\em Mathematical Programming}, pages 1--26, 2021.

\end{thebibliography}
\end{document}